\documentclass[12pt]{amsart}
\usepackage[margin=1in]{geometry}
\usepackage{amsmath,amssymb,amsthm,graphics,mathrsfs}
\usepackage{bbm, stmaryrd, euscript}
\usepackage[all]{xypic}
\usepackage{array}
\usepackage{tikz} \usetikzlibrary{cd}
\usepackage{color}

\usepackage[breaklinks=true, 
colorlinks, linktocpage ]{hyperref}
\usepackage{breakurl}

\numberwithin{equation}{subsection} 
\usepackage[capitalize, nameinlink]{cleveref} 
\crefname{equation}{}{} 
\crefname{subsection}{\S\kern -.5ex}{} 
\crefname{section}{\S\kern -.5ex}{\S\S}
\crefname{enumi}{}{} \creflabelformat{enumi}{#2(#1)#3}

\theoremstyle{plain}
\newtheorem{corollary}[subsection]{Corollary} 
\newtheorem{lemma}[subsection]{Lemma}
\newtheorem{proposition}[subsection]{Proposition}
\newtheorem{theorem}[subsection]{Theorem} 

\theoremstyle{definition}%

\newcommand\CE{{\mathcal E}}

\newcommand\CJ{{\mathscr J}}

\newcommand\CO{{\mathcal O}}

\newcommand\CU{{\mathcal U}}

\newcommand\BBC{{\mathbb C}}

\newcommand\BBh{{\mathbbm h}}
\newcommand\BBL{{\mathbb L}}
\newcommand\BBN{{\mathbb N}}

\newcommand\BBQ{{\mathbb Q}}
\newcommand\BBZ{{\mathbb Z}}

\newcommand\Chow{\operatorname{CH}}

\newcommand\Hom{{\operatorname{Hom}}}

\newcommand\mult{\operatorname{mult}}

\newcommand\res{\operatorname{res}}

\newcommand\inverse{^{-1}}
\renewcommand\th{{^{\text{th}}}}
\newcommand\id{{id}}

\newcommand\Bbar{\overline B}

\newcommand\BBhhat{\widehat{\BBh}}
\newcommand\bfD{\mathbf D}

\newcommand\chat{\widehat c}

\newcommand\ch{\textrm{ch}}

\newcommand\fhat{\widehat f}

\newcommand\GrC{\textrm{\bf Gr}_{\BBC}}
\newcommand\GrSmC{\textrm{\bf Gr-Sm}_{\BBC}}

\renewcommand\hbar{\overline h}

\newcommand\pt{\operatorname{pt}}

\newcommand\Rhat{\widehat R}
\newcommand\reshat{\widehat {\res}}
\newcommand\rev{^{\textrm{rev}}}

\newcommand\Shat{\widehat S}
\newcommand\SmC{\textrm{\bf Sm}_{\BBC}}

\newcommand\varphihat{\widehat \varphi}

\newcommand\wtilde{\widetilde w}
\newcommand\Wtilde{\widetilde W}

\newcommand\xihat{\widehat \xi}

\title[Leray-Hirsch Theorem]{The Leray-Hirsch Theorem for oriented cohomology
  of flag varieties}

\author[J.M. Douglass]{J. Matthew Douglass}
\address{Division of Mathematical
  Sciences, National Science Foundation, 2415 Eisenhower Ave, Alexandria, VA
  22314, USA}

\email{mdouglas@nsf.gov}

\thanks{The authors would like to thank the anonymous referee for
  insightful comments that led to significant improvements in the paper.
  J.M.~Douglass would like to acknowledge that this material is based
  upon work supported by (while serving at) the National Science Foundation.
  However, any opinion, finding, and conclusions or recommendations expressed
  in this material are those of the author and do not necessarily reflect the
  views of the National Science Foundation.
}

\author[C.~Zhong]{Changlong~Zhong} \address{State University of New York at
  Albany, 1400 Washington Ave, Albany, NY 12222}
\email{czhong@albany.edu}



\begin{document}

\begin{abstract}
  We construct two explicit Leray-Hirsch isomorphisms for torus equivariant
  oriented cohomology of flag varieties and give several applications. One
  isomorphism is geometric, based on Bott-Samelson classes. The other is
  algebraic, based on the description of the torus equivariant oriented
  cohomology of a flag variety as the dual of a formal affine Demazure
  algebra.
\end{abstract}

\maketitle
\allowdisplaybreaks


\section{Introduction}

\subsection{}\label{ssec:1.1}
The Leray-Hirsch Theorem is a fundamental result from algebraic topology.
Suppose $\pi\colon E\to B$ is fibre bundle with fibre $F$ and fibre embedding
$i\colon F\to E$. Then $\pi$ and $i$ induce maps
\[
  \pi^*\colon H^*(B)\to H^*(E) \quad\text{and}\quad i^*\colon H^*(E)\to H^*(F)
\]
in singular cohomology (with integer coefficients). Given any homomorphism
$j\colon H^*(F)\to H^*(E)$, we can define a \emph{Leray-Hirsch homomorphism}
\begin{equation*}
  \label{eq:0}
  \varphi_j\colon H^*(B)\otimes H^*(F) \to H^*(E)\quad\text{by}\quad
  \varphi_j(b \otimes f) = \pi^*(b) \smile j(f) . 
\end{equation*}
Under reasonable hypotheses, which include that the map $i^*$ is surjective
and that the map $j$ is a right inverse (or section) of $i^*$, the conclusion
of the Leray-Hirsch Theorem is that the Leray-Hirsch homomorphism $\varphi_j$
is an isomorphism.

\subsection{}\label{ssec:1.2}
Examples with applications in Schubert calculus and the representation theory
of groups of Lie type are the fibre bundles that arise as projections from a
complete flag variety to a partial flag variety. Suppose $G$ is a connected,
reductive, complex algebraic group, $B$ is a Borel subgroup of $G$, and $P$ is
a parabolic subgroup that contains $B$. Then the natural projection $\pi\colon
G/B \to G/P$ is a fibre bundle with fibre $P/B$. Suppose $T\subseteq B$ is a
maximal torus and $L$ is the Levi factor of $P$ that contains $T$. Then
$P/B\cong L/ (B\cap L)$ and the Leray-Hirsch Theorem provides an explicit
isomorphism
\begin{equation*}
  \label{eq:3}
  H^*(G/P)\otimes H^*(L/(B\cap L))\cong H^*(G/B).
\end{equation*}

This isomorphism shows that cohomologically, the flag variety $G/B$ is the
product of the partial flag variety $G/P$ and the smaller flag variety
$L/(B\cap L)$, and is a topological incarnation of the factorization of the
group algebra of the group ring of $W$,
\[
  \BBZ W^L \otimes \BBZ W_L \cong \BBQ W,
\]
where $W$ is the Weyl group of $(G,T)$, $W_L$ is the Weyl group of $(L,T)$,
and $W^L$ is the set of minimal length left coset representatives of $W_L$ in
$W$.

\subsection{}\label{ssec:1.4}
In this paper, we construct two distinct Leray-Hirsch homomorphisms,
\begin{equation*}
  \label{eq:lhiso}
  \BBh_T(G/P)\otimes_{\BBh_T(\pt)}\BBh_T(P/B) \xrightarrow{\ \cong\ }
  \BBh_T(G/B), 
\end{equation*}
where, $\pt$ denotes the one-point variety and $\BBh_\star(\cdot)$ is an
equivariant oriented cohomology theory, as defined by Calm\`es, Zainoulline,
and Zhong. The first homomorphism, $\varphi_g$, which we call the
\emph{geometric Leray-Hirsch homomorphism,} is an extension of the geometric
construction in ordinary cohomology sketched below. It depends on the choice
of resolutions of the Schubert varieties in $G/B$. The second, $\varphi_a$,
which we call the \emph{algebraic Leray-Hirsch homomorphism,} is simpler and
more natural from a purely algebraic perspective, but a direct connection with
the underlying geometry of $G/B$ is not known, and in general, $\varphi_a$ is
only defined when $\BBh_\star(\cdot)$ has a universal property (in a category
of equivariant oriented cohomology theories), namely that $\BBh_\star(\cdot)$
is Chern complete over the point for $T$.

It is shown in \cref{thm:main} that both homomorphisms are isomorphisms when
$\BBh_\star(\cdot)$ is Chern complete, in \cref{thm:noet} that $\varphi_g$ is
an isomorphism whenever $\BBh_{\langle e \rangle}(\pt)$ is Noetherian
($\langle e\rangle$ is the trivial group), and in \cref{sec:apdx} that
$\varphi_g$ is an isomorphism in (algebraic) equivariant $K$-theory.

\subsection{}
In singular cohomology, the fundamental classes of the Schubert varieties
$\overline{B\cdot wB}$, for $w \in W$, form a basis of the homology of
$G/B$. Let $\{\, \xi^w\mid w\in W\,\}$ be the dual basis of $H^*(G/B)$, with
respect to the usual scalar pairing between homology and cohomology, so
$\xi^w\in H^{2\ell(w)}(G/B)$, where $\ell(w)$ is the length of $w$. Replacing
$G/B$ by $L/(B\cap L)$ and using the natural isomorphism $L/(B\cap L) \cong
P/B$ gives a basis $\{\, \xi^v_L\mid v\in W_L\,\}$ of $H^*(P/B)$. Starting
with the $B$-orbits in $G/P$, the same construction gives a basis $\{\,
\xi^w_P\mid w\in W^L\,\}$ of $H^*(G/P)$. Bernstein, Gelfand, and Gelfand
\cite{bernsteingelfandgelfand:schubert} have shown that $\pi^*\colon
H^*(G/P)\to H^*(G/B)$ is injective. In addition, one can check that $i^*\colon
H^*(G/B)\to H^*(P/B)$ is surjective with
\[
  i^*(\xi^w) = \begin{cases} \xi_{L}^w, &w\in W_L, \\ 0, &w\notin W_L. 
     \end{cases}
\]
Define
\[
  j\colon H^*(P/B)\to H^*(G/B)\quad \text{by}\quad \text{$j(\xi_{L}^v) =
    \xi^v$ for $v\in
  W_L$.}
\]
 Then $j$ is a right inverse of $i^*$ and using the results in
\cite{bernsteingelfandgelfand:schubert} one can show that the mapping
\begin{equation}
  \label{eq:bgg}
  H^*(G/P) \otimes H^*(P/B) \to H^*(G/B) \quad\text{given by}\quad \xi^w_P
  \otimes \xi^v_L \mapsto \pi^*(\xi^w_P) \smile j(\xi^v_L) 
\end{equation}
is an isomorphism. 

More generally, one can replace singular cohomology by $T$-equivariant
cohomology, where $T$ is a maximal torus in $B$. Using GKM theory, Drellich
and Tymoczko \cite{drellichtymoczko:module} have described a choice of a right
inverse $j$, constructed a Leray-Hirsch homomorphism,
\begin{equation}
  \label{eq:dt}
  H_T^*(G/P, \BBC) \otimes H_T^*(P/B, \BBC) \to H_T^*(G/B, \BBC)
  \quad\text{given by}\quad  \psi_P^w \otimes \psi_L^v \mapsto \pi^*(\psi_P^w)
  \smile j(\psi_L^v),
\end{equation}
and shown that it is an isomorphism of $H_T(\pt, \BBC)$-modules. In contrast
with the homomorphism in \cref{eq:bgg}, where the classes $\xi_P^w$ and
$\xi_L^v$ arise from Schubert varieties, $\overline{B\cdot wP}$ and
$\overline{(B\cap L)\cdot v(B\cap L)}$, respectively, in \cref{eq:dt} the
classes $\psi_P^w$ and $\psi_L^v$ arise from the \emph{opposite} Schubert
varieties, $\overline{B^-\cdot wP}$ and $\overline{(B\cap L)^-\cdot v(B\cap
  L)}$, respectively.

The Leray-Hirsch isomorphism in \cref{eq:bgg} is obtained from the geometric
Leray-Hirsch homomorphism defined in \cref{ssec:glhh} when $\BBh_\star(\cdot)$
is equivariant cohomology (with integer coefficients), by applying the base
change $\BBZ \otimes_{H_T^*(\pt)}(\, \cdot\, )$, where $\BBZ$ is identified
with the quotient of $H_T^*(\pt)$ by the augmentation ideal (see
\cref{ssec:neq}). The Leray-Hirsch isomorphism in \cref{eq:dt} is the
algebraic Leray-Hirsch homomorphism defined in \cref{ssec:alhh} when
$\BBh_\star(\cdot)$ is equivariant cohomology (with complex coefficients).

In general, the geometric and algebraic Leray-Hirsch homomorphisms are
different. For example, Drellich and Tymoczko show that the base change
functor, $\BBC \otimes_{H_T^*(\pt, \BBC)}(\, \cdot\, )$, where $\BBC$ is
identified with the quotient of $H_T^*(\pt, \BBC)$ by the augmentation ideal,
applied to \cref{eq:dt}, leads to a Leray-Hirsch isomorphism in singular
cohomology that is not equal to the geometric isomorphism in \cref{eq:bgg}.

\subsection{}\label{ssec:1.5}

The genesis of the paper was an effort to prove a Leray-Hirsch isomorphism for
flag varieties in $T$-equivariant $K$-theory. Such an isomorphism would then
lead to Leray-Hirsch isomorphisms
\begin{enumerate}
\item in non-equivariant $K$-theory by base change
  \cite[Rmk.~3.27]{kostantkumar:equivariant},
\item in $T$-equivariant cohomology by passing to associated graded algebras
  \cite[Prop.~2.30]{kostantkumar:equivariant}, and
\item in singular cohomology by base change \cite[\S5]{kostantkumar:nil}.
\end{enumerate}

Using tools developed in \cite{kostantkumar:equivariant} and
\cite{grahamkumar:positivity}, one can construct such a geometric Leray-Hirsch
homomorphism in equivariant $K$-theory and show that it is an
isomorphism. Much of this argument depends only on formal algebraic properties
of equivariant $K$-theory. Calm\`es, Lenart, Zainoulline, Zhong, and others
have developed an algebraic description of the $T$-equivariant oriented
cohomology theory of partial flag varieties that generalizes the constructions
and results of Kostant and Kumar. In this algebraic approach it is natural to
treat the geometric and algebraic constructions simultaneously, as is done in
this paper.

However, this greater generality comes at a cost. The algebraic constructions
depend only on a root datum and a formal group law, and the argument that these
constructions give a purely algebraic construction of $\BBh_T(G/P)$ has only
been completed in the case when the equivariant oriented cohomology theory
$\BBh_\star(\cdot)$ is ``Chern complete over the point for $T$." In other
words, $\BBh_T(\pt)$ is separated and complete with respect to the
$\gamma$-filtration (see \cite[\S2]  {calmeszainoullinezhong:equivariant}).

Some classical equivariant oriented cohomology theories, such as equivariant
Chow rings and equivariant $K$-theory, are not Chern complete over the point
for $T$.  On the other hand, if $\BBh_\star(\cdot)$ is any equivariant
cohomology theory with $\BBh_{\langle e \rangle} (\cdot)$ equal to the
universal oriented cohomology theory of Levine and Morel, namely algebraic
cobordism, then $\BBh_\star(\cdot)$ is Chern complete over the point for
$T$. In general, one can compose the functor $\BBh_\star(\, \cdot\,)$ with a
``base change'' functor to get an equivariant cohomology theory that is Chern
complete over the point for $T$ (see \cite[Rmk.~2.3]
{calmeszainoullinezhong:equivariant} and \cref{ssec:comp}).

\subsection{}\label{ssec:1.8}
The rest of this paper is organized as follows. In \cref{sec:lhh} we review
the notions of oriented and equivariant oriented cohomology theories for
smooth, complex varieties from \cite{levinemorel:algebraic} and
\cite{calmeszainoullinezhong:equivariant}; and construct the geometric
Leray-Hirsch homomorphism for the $T$-equivariant fibre bundle $\pi\colon
G/B\to G/P$ under the (mild) standing assumptions in \cref{ssec:asm}. In
\cref{sec:lhi} we assume that $\BBh_\star(\cdot)$ is Chern complete over the
point for $T$ and review the algebraic description of $\BBh_T(G/B)$,
$\BBh_T(G/P)$, $\pi^*$, and $\pi_*$ developed in
\cite{calmeszainoullinezhong:equivariant},
\cite{calmeszainoullinezhong:coproduct}, and
\cite{calmeszainoullinezhong:push}; define the algebraic Leray-Hirsch
homomorphism; and state the main theorem, namely that the geometric and
algebraic Leray-Hirsch homomorphisms are isomorphisms when $\BBh_\star(\cdot)$
is Chern complete over the point for $T$. The main theorem is proved in
\cref{sec:pf}. In \cref{sec:app} we give several applications, including in
particular that the geometric Leray-Hirsch homomorphism is an isomorphism when
$\BBh_{\langle e\rangle}(\pt)$ is Noetherian. This includes equivariant Chow
rings and equivariant (algebraic) $K$-theory as special cases. Finally, in the
appendix we give a direct proof that the geometric Leray-Hirsch homomorphism
is an isomorphism in $T$-equivariant $K$-theory.

\subsection{}\label{ssec:1.9}

Throughout this paper, for convenience we consider only complex algebraic
groups and complex varieties. The results and proofs are unchanged if the
complex field is replaced by any algebraically closed field with
characteristic zero. If $X$ is a variety, the morphism from $X$ to a point
will be denoted by $a$ or $a_X$ depending on context, so $a=a_X\colon X \to
\pt$.

Unless otherwise indicated, the identity element in a multiplicative group is
denoted by $e$. Any ambient group should always be easily determined by
context. In addition, we assume that the reductive group $G$ is in fact
semisimple and we fix a maximal torus $T$. Passing to a reductive quotient,
say $\widetilde G$, of $G$, and a maximal torus $\widetilde T$ of $\widetilde
G$, changes the coefficient ring, $\BBh_\star(\pt)$, but does not change the
varieties $G/B$, $G/P$, and $P/B$, or the structure of their equivariant
cohomology rings as $\BBh_\star(\pt)$-modules (see \cite[Lem.~11.1]
{calmeszainoullinezhong:equivariant}).

Further standing assumptions are collected in \cref{ssec:asm} after the
requisite notation has been developed.

\section{The geometric Leray-Hirsch homomorphism} \label{sec:lhh}

In this section we fix notation and review the concepts from (algebraic)
equivariant oriented cohomology theories needed to give the precise
formulation of the geometric Leray-Hirsch homomorphism.

\subsection{Group theory notation}\label{ssec:gp}

As in the introduction, $G$ is a semisimple, complex algebraic group,
\[
  T\subseteq B \subseteq P\subseteq G
\]
are a fixed maximal torus, Borel subgroup, and parabolic subgroup of $G$,
respectively; $\pi\colon G/B\to G/P$ is the projection, $i\colon P/B \to G/B$
is the inclusion, and $L$ is the Levi factor of $P$ that contains $T$. Then
$B\cap L$ is a Borel subgroup of $L$ and we may identify $P/B$ with the flag
variety of $L$ via the natural projection $L/(B\cap L) \xrightarrow{\,
  \cong\,} P/B$.

Let $\Lambda$ denote the character group of $T$ and let $\Phi$ be the root
system of $(G,T)$, so $\Phi$ is a subset of $\Lambda$. Let $\Phi^+$ be the set
of positive roots, defined as the set of weights of $T$ on the cotangent space
$T^*_{eB} (G/B)$ at the point $eB\in G/B$. Then $-\Phi^+$ is the set of
weights of $T$ on the tangent space $T_{eB} (G/B)$. Let $\Phi_L$ be the subset
of roots corresponding to the root system of $(L,T)$, and let $\Phi_L^+=
\Phi_L\cap \Phi^+$.

Let $W$ be the Weyl group of $(G,T)$ and let $W_L$ be the Weyl group of
$(L,T)$. We always identify $W_L$ as a subgroup of $W$. The positive system
$\Phi^+$ determines a base of $\Phi$ that in turn determines a set of simple
reflections in $W$, a length function, $\ell$, on $W$, and the Bruhat order,
$\leq$, on $W$. Set $W^L=\{\, w\in W\mid \forall s\in W_L, ws>s\,\}$. Then
$W^L$ is the set of minimal length right coset representatives of $W_L$ in
$W^L$ and multiplication in $W$ defines a bijection $W^L \times
W_L\leftrightarrow W$.

For the simple reflection, $s=s_\alpha\in W$, let $P_s= \langle B, sT\rangle$
be the parabolic subgroup of $G$ that is generated by $B$ and the coset
$sT$. Then $P_s$ contains $B$ and $T$ acts on $T^*_{eB} (P_s/B)$ as
$\alpha$. Let $L_s$ be the Levi factor in $P_s$ that contains $T$. Then
$W_{L_s}= \{e,s\}$ and for $z\in W$, $z\in W^{L_s}$ if and only if $zs>z$.

Fix linear orders on $W^L$ and $W_L$ that extend the Bruhat order. Then the
lexicographic order on $W$ given by the factorization $W= W^LW_L$, say
$\preceq$, is a linear order on $W$ that extends the Bruhat order \cite[Lemma
3.5]{deodhar:characterizations}. In this linear order on $W$ we have
\[
  \text{$wv \preceq w'v'$ \quad if and only if \quad $w \preceq w'$, or $w=w'$
    and $v \preceq v'$.}
\]
When we consider square matrices indexed by $W$ it will always be with respect
to this linear order.

\subsection{Notation for tuples (of simple reflections)}
Suppose $I=(s_1, \dots, s_p)$ and $J=(s_1', \dots, s_q')$ are two tuples of
simple reflections in $W$.
\begin{itemize}
\item $I\rev= (s_p, \dots, s_1)$ is the reverse of $I$.
\item If $s$ is any simple reflection in $W$, then write $s\in I$ if $s=s_i$
  for some $1\leq i\leq p$.
\item Write $I\sqsubseteq J$ if $I$ is a subsequence of $J$.
\item $I\amalg J$ denotes the concatenation of $I$ and $J$.
\end{itemize}

Suppose $E$ and $F$ are subsequences of $I$. 
\begin{itemize}
\item $E\setminus F$ is the subsequence of $I$ formed by the entries in $I$
  that are in positions occupied by $E$ and not $F$.
\item $E\sqcap F$ is the subsequence of $I$ formed by the entries in $I$ that
  are in positions occupied by both $E$ and $F$.
\item $E\sqcup F$ is the subsequence of $I$ formed by the entries in $I$ that
  are in positions occupied by either $E$ or $F$.
\end{itemize}

\subsection{(Algebraic) oriented cohomology theories and formal group laws
  \cite[\S1.1, \S1.2] {levinemorel:algebraic}} \label{ssec:oct}

Let $\SmC$ denote the category of smooth, quasi-projective, complex, algebraic
varieties.  An algebraic oriented cohomology theory, or simply an
\emph{oriented cohomology theory,} in the sense of Levine and Morel, is a
contravariant functor, $\BBh=\BBh(\cdot)$, from $\SmC$ to the category of
commutative, graded rings with identity together with a natural push-forward
map $f_*\colon \BBh(X) \to \BBh(Y)$ for each projective morphism $f\colon X\to
Y$ in $\SmC$. The functor $\BBh$ and the push-forward maps are assumed to
satisfy a natural list of axioms.

Let $\BBh$ be an oriented cohomology theory and set $R=\BBh(\pt)$. It follows
from the axioms that one can define the notion of $\BBh$-Chern classes of
vector bundles. If $E$ is a non-zero vector bundle over a smooth variety $X$,
then $c_1(E)\in \BBh(X)$ denotes the first Chern class of $E$. With this
notation, there is a formal power series $F\in R[[t,u]]$ such that
$c_1(L\otimes M)= F(c_1(L), c_1(M))$ for line bundles $L$ and $M$ on $X$. The
formal power series $F$ has the form
\begin{equation}
  \label{eq:21}
  F(t,u)= \sum_{i,j} r_{i,j} t^i u^j = t+u-tu\cdot G(t,u),
  \quad\text{where}\quad G(t,u) = -r_{1,1} -r_{2,1}t -r_{1,2}u - \dotsm.  
\end{equation}
More precisely, $F$ is a commutative formal group law of rank one, or more
simply a formal group law. In other words,
\[
  F(t,0)= F(0,t)=t, \quad F(t,u)= F(u,t),\quad \text{and}\quad F(F(t,u),v)=
  F(t, F(u,v)),
\]
and $F$ is called the formal group law of $\BBh$.

\subsection{Examples}\label{ssec:fglex}

The fundamental example of an oriented cohomology theory is given by the Chow
ring, denoted by $\Chow^*$. If $X$ is a smooth, quasi-projective variety and
$L$ and $M$ are line bundles on $X$, then $c_1(L\otimes M) = c_1(L)+ c_1(M)$,
so the formal group law of $\Chow^*$ is the \emph{additive formal group law:}
\[
  F_{add}(t,u)= t+u \in \BBZ[[t,u]].
\]
Another example of an oriented cohomology theory with formal group law equal
to $F_{add}$ is the even part of de Rahm cohomology.

If $R$ is any ring with identity and $\kappa \in R\setminus \{0\}$, then
\[
  F(t,u)= t+u-\kappa tu \in R[[t,u]]
\]
is a formal group law, called a \emph{multiplicative formal group law.} For
example, the functor $K(\cdot)$, where $K(X)$ is the Grothendieck group of
vector bundles on the smooth variety $X$, is an oriented cohomology theory
whose formal group law is multiplicative with $\kappa=1$, namely $F(t,u)= t+u
-tu\in \BBZ[[t,u]]$.

Suppose $R$ is an integral domain. One checks that a formal group law,
$F(t,u)\in R[[t,u]]$, is a polynomial if and only if $F(t,u)= t+u-\kappa tu$
for some $\kappa\in R$, so if and only if $F$ is the additive group law or a
multiplicative formal group law.

A universal example of an oriented cohomology theory is the algebraic
cobordism functor, $\Omega^*$, constructed by Levine and Morel. The ring
$\Omega^*(\pt)$ is isomorphic to the Lazard ring, $\BBL= \BBZ[a_{i,j}\mid
i,j\in \BBN]/J$, where the $a_{i,j}$ are indeterminates and $J$ is the ideal
generated by the relations among the variables $a_{i,j}$ imposed by the
requirement that if $f_{i,j}= a_{i,j}+ J$, then $F(t,u) = \sum_{i,j} f_{i,j}
t^i u^j$ defines a formal group law in $\BBL$. Moreover, the formal group law
of $\Omega^*(\pt)$ is the \emph{universal formal group law} on $\BBL$:
\[
  F_{univ}(t,u) = \sum_{i,j} f_{i,j} t^i u^j.
\]
The oriented cohomology theory $\Omega^*$ has the universal property that if
$\BBh$ is any oriented cohomology theory, then there is a unique natural
transformation $\Omega^* \to \BBh$ that commutes with push-forwards of
projective morphisms.

Suppose $R$ is a commutative graded ring with identity and $F$ is a formal
group law on $R$. Define $\BBh$ to be the composition of $\Omega^*$ with the
base change functor $R \otimes_{\BBL} (\cdot)$, where $R$ is an $\BBL$-algebra
via the unique ring homomorphism $\psi\colon \BBL\to R$ such that $F(t,u)=
\sum_{i,j} \psi(f_{i,j}) t^i u^j$. Then $\BBh$ is an oriented cohomology
theory with $\BBh(\pt)= R \otimes_{\BBL}\BBL \cong R$ and formal group law $F$
\cite[1.1]{levinemorel:algebraic}.

\subsection{(Algebraic) equivariant oriented cohomology theories \cite[\S2]
  {calmeszainoullinezhong:equivariant}}

Let $\GrC$ be the subcategory of $\SmC$ consisting of linear algebraic groups
and group homomorphisms. Define $\GrSmC$ to be the category with objects all
pairs $(H,X)$, where $H\in \GrC$, $X\in \SmC$, and $X$ is a $H$-variety. Given
$(H,X)$ and $(H',X')$, a morphism in $\GrSmC$ is a pair $(\varphi,f)$, where
$\varphi\colon G\to G'$ is a morphism in $\GrC$, $f\colon X\to X'$ is a
morphism in $\SmC$, and the obvious compatibility condition intertwining the
$H$-action on $X$ and the $H'$-action on $X'$ holds, namely $f(g\cdot x)=
\varphi(g)\cdot f(x)$ for $(g,x)\in H\times X$ (or equivalently, the
corresponding diagram commutes).

An \emph{equivariant oriented cohomology theory} is an additive, contravariant
functor from the category $\GrSmC$ to the category of commutative rings with
identity, endowed with push-forward maps for equivariant projective morphisms
(for a fixed algebraic group), restriction maps for morphisms of algebraic
groups (for a fixed smooth variety), and a ``total characteristic class'' map
(for a fixed algebraic group), that satisfy a collection of natural axioms
\cite[\S2] {calmeszainoullinezhong:equivariant}.

Suppose $\BBh_\star(\cdot)$ is an equivariant oriented cohomology theory. 
\begin{itemize}
\item On objects, the functor $\BBh_\star(\cdot)$ carries $(H, X)$ to the
  commutative ring $\BBh_H(X)$.
\item The subcategory of $\GrSmC$ with objects whose first component is equal
  to a fixed group $H$ with morphisms $(\id_H, f)$ is canonically isomorphic
  to the category of $H$-varieties. The restriction of $\BBh_\star(\cdot)$ to
  this subcategory defines a contravariant functor, denoted by
  $\BBh_H(\cdot)$, from the category of $H$-varieties to the category of
  commutative rings with identity. If $f\colon X\to Y$ is an $H$-equivariant
  morphism of $H$-varieties, then $\BBh_{H}(f)$ is denoted simply by $f^*$, so
  $f^*\colon \BBh_H(Y) \to \BBh_H(X)$ is a ring homomorphism. If $f$ is
  projective, then the push-forward map is denoted simply by $f_*$, so
  $f_*\colon \BBh_H(X) \to \BBh_H(Y)$. The mapping $f_*$ is assumed to be
  $\BBh_*(Y)$-linear, but not a ring homomorphism.

\item For a homomorphism of algebraic groups $\psi\colon H\to H'$, let
  $\res_\psi$ be the ``restriction'' functor from $H'$-varieties to
  $H$-varieties, where $H$ acts on a $H'$-variety via the pull-back through
  $\psi$. Then $(\psi, \id_X)\colon (H,X) \to (H',X)$ is a morphism in
  $\GrSmC$ and $\BBh(\psi, \id_X)\colon \BBh_{H'}(X) \to
  \BBh_{H}(\res_\psi(X))$. Abusing notation slightly, set $\res_\psi =
  \BBh(\psi, \id_X)$. These are the restriction maps in the definition of
  $\BBh_\star(\cdot)$.

\item By definition, $\BBh_\star(\cdot)$ comes equipped with ``total
  characteristic classes:'' For $(H,X)\in \GrSmC$, there is a group
  homomorphism
  \[
    c^H\colon K_H(X)\to \CU\big(\BBh_H(X)[t]\big)
  \]
  where $K_H(X)$ is considered as an additive group, $\BBh_H(X)[t]$ is the
  polynomial ring in the variable $t$, and $\CU\big(\BBh_H(X)[t]\big)$ is the
  group of units in $\BBh_H(X)[t]$. The group homomorphism $c^H$ is assumed to
  be natural in $X$, and for the class of a $H$-equivariant vector bundle, say
  $\CE\in K_H(X)$, the coefficient of $t^k$ in $c^H(\CE)$ is denoted by
  $c^H_m(\CE)$ and is called the $k\th$ $H$-equivariant characteristic class.

\item When $H=\langle e\rangle$, the full subcategory of smooth $\langle
  e\rangle$-equivariant varieties is canonically isomorphic to $\SmC$. One of
  the axioms is that $\BBh_{\langle e \rangle}(\cdot)$ is an oriented
  cohomology theory in the sense of Levine and Morel, but with values in the
  category of commutative rings, not commutative graded rings. By definition,
  the formal group law of $\BBh_\star(\cdot)$ is the formal group law of
  $\BBh_{\langle e\rangle}(\cdot)$, and the $\BBh_{\langle
    e\rangle}(\cdot)$-Chern classes are identified with the $\langle
  e\rangle$-equivariant characteristic classes.
  
\item Let $\imath\colon \langle e\rangle \to T$ denote the inclusion. Since
  the composition $\langle e\rangle \xrightarrow{\, \imath\,} T
  \xrightarrow{\, a\,} \langle e\rangle$ is the identity and
  $\res_\imath\res_a =\res_{a\imath}$, it follows that $\res_a$ is injective
  and $\res_\imath$ is surjective, for any $X$. For example, $\res_a\colon
  \BBh_{\langle e\rangle}(\pt) \to \BBh_T(\pt)$ is injective and $\res_\imath
  \colon \BBh_T(\pt)\to \BBh_{\langle e\rangle}(\pt)$ is surjective.
   
\end{itemize}

\subsection{Standing assumptions}\label{ssec:asm}

A reduced expression for an element $z\in W$ is a sequence $I=(s_1, \dots,
s_p)$ of simple reflections such that $z=s_1\dotsm s_p$ and $\ell(z)=p$. In
the rest of this paper we fix a set of ``$L$-compatible'' reduced expressions
for the elements of $W$, that is we choose a fixed reduced expression, $I_w$,
for each $w\in W^L$ and a fixed reduced expression, $I_v$, for each $v\in
W_L$, and define $I_{wv}= I_w \amalg I_v$ to be the concatenation of $I_w$ and
$I_v$. It follows from \cite[p.37 Exer.~3]{bourbaki:groupes} that $\{\,
I_{wv}\mid w\in W^L, v\in W_L\,\}$ is a set of reduced expressions for the
elements of $W$.

For the rest of this paper $\BBh_\star(\cdot)$ denotes an equivariant oriented
cohomology theory,
\[
  \text{$\BBh_H= \BBh_H(\pt)$ for any algebraic group $H$,} \quad
  R=\BBh_{\langle e\rangle}, \quad \text{and}\quad S=\BBh_T.
\]
We use $R$ and $S$ to help simplify the notation -- $R$ is the coefficient
ring for the oriented cohomology theory $\BBh_{\langle e\rangle} (\cdot)$ and
$S$ is the coefficient ring for the $T$-equivariant theory $\BBh_T
(\cdot)$. The ring homomorphism $\res_a\colon R\to S$ is injective and
determines an $R$-module structure on $S$, and the ring homomorphism
$\res_i\colon S\to R$ is surjective and determines an $S$-module structure on
$R$. Unless otherwise indicated, we consider $S$ and $R$ with these module
structures.

We assume that either $2$ is not a zero divisor in $R$, or that the derived
group of $G$ does not contain an irreducible factor of type $C$. With this
assumption we can freely use the results in
\cite{calmeszainoullinezhong:coproduct} and \cite{calmeszainoullinezhong:push}
as needed.

\subsection{Completion}\label{ssec:comp}

As explained in \cite[\S2] {calmeszainoullinezhong:equivariant}, for a smooth
$H$-variety, $X$, the $H$-equivariant characteristic classes define the
``$\gamma$-filtration'' in $\BBh_H(X)$, as usual in the theory of
$\lambda$-rings, in which the $k\th$ term is denoted by $\gamma^k\BBh_H(X)$.

Let $\widehat{\BBh_H}$ denote the completion of the ring $\BBh_H$ with respect
to the $\gamma$-filtration. For example, it follows from \cite[Lem.~1.1.3]
{levinemorel:algebraic} that $\gamma^0 \BBh_{\langle e\rangle} = \BBh_{\langle
  e\rangle}$ and $\gamma^1 \BBh_{\langle e\rangle} = 0$, and so
\begin{equation}
  \label{eq:he}
  \widehat R = \widehat{\BBh_{\langle e\rangle}} = \BBh_{\langle e\rangle}=
  R.  
\end{equation}
It is shown in \cite[\S2] {calmespetrovzainoulline:invariants} that if
$\{\omega_1, \dots, \omega_n\}$ is a basis of $\Lambda$, then
\[
  \Shat= \widehat{\BBh_T} \cong \widehat{\BBh_{\langle e\rangle}}
  [[x_{\omega_1}, \dots, x_{\omega_n}]] = R[[x_{\omega_1}, \dots,
  x_{\omega_n}]]
\]
is a formal power series ring (see \cref{ssec:S+}).

Now define a contravariant functor, $\BBhhat_\star(\cdot)$, from $\GrSmC$ to
the category of commutative rings on objects by ``base change:'' For a smooth
$H$-variety $X$,
\[
  \BBhhat_H(X) = \widehat{\BBh_H} \otimes_{\BBh_H} \BBh_H(X).
\]
\begin{itemize}
\item If $f\colon X\to Y$ is an $H$-equivariant morphism of $H$-varieties,
  define $\fhat^*= \id\otimes f^*$. If $f$ is projective, define $\fhat_*=
  \id\otimes f_*$. Then
 \[
   \fhat^*\colon \BBhhat_H(Y)\to \BBhhat_H(X) \quad\text{and}\quad
   \fhat_*\colon \BBhhat_H(X)\to \BBhhat_H(Y).
 \]
  
\item If $\psi\colon H\to H'$ is a homomorphism of algebraic groups and $X'$
  is a $H'$-variety, define $\reshat_\psi= \id\otimes \res_\psi$. Then
  \[
    \reshat_\psi \colon \BBhhat_{H'}(X')\to \BBhhat_H(\res_\psi(X')) =
    \BBhhat_H(X').
  \]
  
\item The inclusion $\BBh_H(X) \cong 1\otimes \BBh_H(X)\subseteq \BBhhat_H(X)$
  extends naturally to a ring homomorphism $\BBh_H(X)[t] \to \BBhhat_H
  (X)[t]$, which in turn restricts to a group homomorphism $\rho\colon
  \CU\big(\BBh_H(X)[t] \big) \to \CU\big( \BBhhat_H (X)[t] \big)$.  Define
  $\chat^H=\rho \circ c^H$. Then
  \[
    \chat^H\colon K_H(X) \to \CU\big( \BBhhat_H (X)[t] \big) .
  \]
\end{itemize}

The next lemma is a stronger version of \cite[Rmk.~2.3]
{calmeszainoullinezhong:equivariant}.


\begin{lemma}\label{lem:RRhat}
  The functor $\BBhhat_\star(\cdot)$ endowed with the push-forward maps,
  $\fhat_*$, for equivariant projective morphisms, restriction maps,
  $\reshat_\psi$, for morphisms of algebraic groups, and total characteristic
  class homomorphisms, $\chat^H$, is an equivariant oriented cohomology theory
  with $\BBhhat_{\langle e\rangle}(\cdot)= \BBh_{\langle e\rangle}(\cdot)$
  that is Chern complete over the point for $T$.
\end{lemma}

\begin{proof}
  It is straightforward to check that the axioms all hold. The assertion that
  $\BBhhat_{\langle e\rangle}(\cdot)= \BBh_{\langle e\rangle}(\cdot)$ follows
  from \cref{ssec:comp}\cref{eq:he}
\end{proof}

\subsection{Bott-Samelson resolutions and cohomology classes
  \cite[\S7,8]{calmeszainoullinezhong:equivariant}}\label{ssec:bs}

For a tuple $I=(s_1, \dots, s_p)$ of simple reflections such that $s_p\notin
W_L$, define
\[
  Z^P_I = P_{s_1} \times^B \dotsm \times^B P_{s_{p-1}} \times^B (P_{s_p}P /P)
\]
and let
\[
  q^P_I\colon Z^P_I\to G/P
\]
be the projection given by multiplication of the factors. Notice that because
$s_p\notin W_L$, the natural projection $P_{s_p}/B \to P_{s_p}P/P$ is an
isomorphism. It follows that the natural projection from $Z^B_{I}$ to
$Z^P_{I}$ is an isomorphism and that the diagram
\[
  \xymatrix{ Z^B_{I} \ar[r]_{\cong}^-{\pi_I} \ar[d]^{q^B_{I}} & Z^P_{I}
    \ar[d]^{q^P_{I}} \\ G/B \ar[r]^{\pi} &G/P}
\]
commutes. 

For an algebraic subgroup, $H\subseteq T$, define the \emph{Bott-Samelson
  class}
\[
  \xi^P_{I.H} = \pi_* (q^B_{I})_*(1) = (q^P_{I})_* (\pi_I)_*(1)\in
  \BBh_H(G/P),
\]
where $1$ denotes the identity in $\BBh_H(Z^B_{I})$. If $P=B$, then there is
no condition on $I$ and we set $\xi_{I,H}= \xi_{I,H}^B$. If $s_p\notin W_L$,
then $\xi_{I,H}$ and $\xi^P_{I,H}$ are both defined and $\pi_*(\xi_{I,H})=
\xi^P_{I,H}$. If the ambient group is clear from context, it is not be
included in the notation.

If $I=I_w$ for some $w\in W$, then we write simply $w$ instead of $I_w$ in the
notation for Bott-Samelson classes. Thus,
\[
  \xi^P_w = \xi^P_{I_w} = \xi^P_{I_w,H} \quad\text{and}\quad \xi_w=\xi^B_w
  =\xi^B_{I_w,H}.
 \]
 With this notation, $\pi_{*}(\xi_w)= \xi^P_w$ for $w\in W^L$.

\subsection{Duality and the scalar pairing}\label{ssec:dual}

For an algebraic group, $H$, and a smooth, projective $H$-variety, $X$, the
canonical morphism $a_X\colon X\to \pt$ is projective, so $(a_X)_*\colon
\BBh_H(X) \to \BBh_H(\pt)$ is defined. Define a pairing
\[
  \langle \,\cdot\,,\,\cdot\, \rangle\colon \BBh_H(X) \otimes_S \BBh_H(X)
  \to S \quad\text{by}\quad \langle \xi, \xi'\rangle = (a_X)_*(\xi \xi').
\]
We call this pairing the \emph{scalar pairing} because it reduces to the
scalar product pairing between singular cohomology and singular homology given
by evaluation. The scalar pairing is also called the push-forward pairing.

Note that the scalar pairing,$\langle \,\cdot\,,\,\cdot\, \rangle$, depends on
$\BBh_\star(\cdot)$, $H$, and $X$. We will state the domain of the pairing
when there is any chance of confusion.

\begin{theorem}\label{thm:bssp}
  The variety $G/P$ is ``$T$-equivariantly formal'' in the sense that there is
  a natural $S$-module isomorphism
  \begin{equation*}\label{eq:m}
    f\colon \BBh_T(G/P) \xrightarrow{\ \cong\ } S\otimes_R \BBh_{\langle e
      \rangle}(G/P) \quad\text{such that}\quad f(\xi^P_{w,T})= 1\otimes
    \xi^P_{w, \langle e\rangle} 
  \end{equation*}
  for $w\in W^L$, the Bott-Samelson classes $\{\, \xi^P_w\mid w\in W^L\,\}$
  form an $S$-basis of $\BBh_T(G/P)$, and the scalar pairing on $\BBh_T(G/P)$ is
  non-degenerate.
\end{theorem}

\begin{proof}
  It is shown in \cite[Lem.~11.1] {calmeszainoullinezhong:equivariant} that
  the homomorphism
  \begin{equation}
    \label{eq:c}
    \fhat\colon \Rhat \otimes_{\Shat} \BBhhat_T(G/P) \xrightarrow{\ \ }
    \BBhhat_{\langle e \rangle}(G/P) , 
  \end{equation}
  which satisfies $\fhat(\widehat r\otimes \widehat \zeta)= \widehat r \cdot
  \reshat_i (\widehat\zeta)$, is an $\Rhat$-algebra isomorphism, that the
  Bott-Samelson classes, $\{\, \xihat^P_{w, {\langle e \rangle}} \mid w\in
  W^L\,\}$, form an $\Rhat$-basis of $\BBhhat_{\langle e \rangle} (G/P)$, and
  that the scalar pairing on $\BBhhat_{\langle e \rangle} (G/P)$ is
  non-degenerate.
  
  It was shown in \cref{lem:RRhat} that $\Rhat=R$ and $\BBhhat_{\langle e
    \rangle} (G/P)= \BBh_{\langle e \rangle} (G/P)$, so the Bott-Samelson
  classes, $\{\, \xi^P_{w, \langle e \rangle}\mid w\in W^L\,\}$, form an
  $R$-basis of $\BBh_{\langle e \rangle} (G/P)$. Next, it follows from the
  assumption that restriction homomorphisms commutes with proper push-forwards
  that $\res_i\big(\xi^P_{w,T} \big) = \xi^P_{w, \langle e\rangle}$, so $\fhat
  \big(1\otimes \xihat^P_{w,T} \big)= \xihat^P_{w,\langle
    e\rangle}$. Finally, one checks that $\Rhat \otimes_{\Shat}
  \BBhhat_T(G/P)\cong R \otimes_{S} \BBh_T(G/P)$, and so \cref{eq:c} can be
  identified with the $R$-linear isomorphism
  \begin{equation}
    \label{eq:b}
    R \otimes_{S} \BBh_T(G/P) \xrightarrow{\ \cong \ }
    \BBh_{\langle e \rangle}(G/P) \quad\text{given by}\quad r\otimes
    \xi^P_{w,T}\mapsto r\xi^P_{w, \langle e \rangle}.
  \end{equation}

  Now apply the base change $S\otimes_R -$ to \cref{eq:b} and identify
  $S\otimes_R \big( R\otimes_S \BBh_T(G/P) \big)$ with $\BBh_T(G/P)$ to get
  the $S$-algebra isomorphism in the statement of the theorem
  \begin{equation*}
    \label{eq:x}
    f\colon \BBh_T(G/P) \xrightarrow{\ \cong\ } S\otimes_R \BBh_{\langle e
      \rangle}(G/P) \quad\text{with}\quad f(\xi^P_{w,T})=  1\otimes
    \xi^P_{w, \langle e \rangle}. 
  \end{equation*}
  Since $\{\, \xi^P_{w, \langle e \rangle} \mid w\in W^L\,\}$ is an $R$-basis
  of $\BBh_{\langle e \rangle} (G/P)$, it follows that $\{\, \xi^P_{w, T}\mid
  w\in W^L\,\}$ is an $S$-basis of $\BBh_T(G/P)$. To complete the proof,
  notice that $f$ intertwines the scalar pairing on $\BBh_T(G/P)$ with the
  extension of the scalar pairing on $\BBh_{\langle e \rangle}(G/P)$ to
  $S\otimes_R \BBh_{\langle e \rangle}(G/P)$, because restriction
  homomorphisms, $\res$, commute with push-forwards, $a_*$. It then follows
  that the scalar pairing on $\BBh_T(G/P)$ is non-degenerate, because the
  scalar pairing on $\BBh_{\langle e \rangle}(G/P)$ is non-degenerate.
\end{proof}

\subsection{} Using the theorem, as well as the conventions in \cref{ssec:bs},
define $\{\, \xi_P^w\mid w\in W^L\,\}$ to be the $S$-basis of $\BBh_T(G/P)$
that is dual to the Bott-Samelson basis with respect to the scalar
pairing. Then
\[
  \langle \xi_w^P, \xi_P^{w'} \rangle = \delta_{w,w'} \quad\text{for $w,w'\in
    W^L$.}
\]
As usual, set $\xi^w= \xi_B^w$ for $w\in W$.

\subsection{The flag variety of $L$ and $\BBh_T(P/B)$}\label{ssec:fvL}

Consider the composition
\[
  \xymatrix{L/(B\cap L) \ar[r]_-{\cong} &P/B \ar@{^(->}[r]^-{i} & G/B.}
\]

For $v\in W_L$ we have the smooth variety $Z^{B\cap L}_{I_v}$ and the
projection $q^{B\cap L}_v\colon Z^{B\cap L}_{I_v} \to L/(B\cap L)$. There is a
canonical isomorphism, $Z^{B\cap L}_{I_v} \cong Z^{B}_{I_v}$, and the diagram
\[
  \xymatrix{ Z^{B\cap L}_{I_v} \ar[r]_{\cong} \ar[d]^{q^{B\cap L}_{I_v}}&
    Z^B_{I_v} \ar[d] \ar[dr]^{q^B_{I_v}} & \\
    L/(B\cap L) \ar[r]_-{\cong} &P/B \ar@{^(->}[r]^-{i} & G/B}
\]
commutes. Let $\xi^L_v$ be the Bott-Samelson class in $\BBh_T(P/B)$ determined
by $q^{B\cap L}_{I_v}$, that is, the image in $\BBh_T(P/B)$ of $1 \in
\BBh_T(Z^{B\cap L}_{I_v})$. Then $\{\, \xi^L_v\mid v\in W_L\,\}$ is an
$S$-basis of $\BBh_T(P/B)$ and by the naturality of push-forward maps we have
$i_*(\xi^L_v) = \xi_v$.

Let $\{\, \xi_L^v\mid v\in W_L\,\}$ be the dual basis of $\BBh_T(P/B)$ with
respect to the scalar pairing. If $z\in W$, then it follows by duality and the
projection formula that
\[
  i^*(\xi^z) = \sum_{v\in W_L} \langle i^*(\xi^z), \xi^L_v \rangle_{P/B} \,
  \xi_L^v = \sum_{v\in W_L} \langle \xi^z, i_{*}( \xi^L_v) \rangle_{G/B}
  \,\xi_L^v = \sum_{v\in W_L} \langle \xi^z, \xi_v \rangle_{G/B} \,\xi_L^v .
\]
Therefore, $i^*\colon \BBh_T(G/B) \to \BBh_T(P/B)$ is the projection given by
\begin{equation*}
  \label{eq:7}
  i^*(\xi^z)= \begin{cases} \xi^v_L&\text{if $z=v\in W_L$,}\\
    0&\text{if $z\notin W_L$.} \end{cases}
\end{equation*}

\subsection{The geometric Leray-Hirsch homomorphism}\label{ssec:glhh}

Define
\[
  j_g\colon \BBh_T(P/B) \to \BBh_T(G/B) \quad\text{by $j_g( \xi^v_L) = \xi^v$
    and $S$-linearity.}
\]
Then $j_g$ is a right inverse of the projection $i^*$.

The \emph{geometric Leray-Hirsch homomorphism} is the composition $\varphi_g
=\operatorname{mult} \circ (\pi^*\otimes j_g)$, where $\mult$ is the
multiplication map:
\begin{equation*}
  \label{eq:glhi}
  \vcenter{\vbox{    \xymatrix{ \BBh_T(G/P) \otimes_{S} \BBh_T(P/B)
        \ar[rr]^-{\pi^* \otimes j_g} \ar@/^30pt/[rrrr]^{\varphi_g}
        && \BBh_T(G/B) \otimes_{S}  \BBh_T(G/B) \ar[rr]^-{\operatorname{mult}}
        && \BBh_T(G/B) .} }}  
\end{equation*}
To simplify the notation in the arguments below, define
\[
  \zeta^w=\pi^*(\xi_P^w) \quad\text{for}\quad w\in W^L.
\]
Then
\begin{equation}
  \label{eq:19}
  \varphi_g(\xi_P^w \otimes \xi_L^v)= \pi^*(\xi_P^w) j_g( \xi_L^v) =\zeta^w
  \xi^v   \quad\text{for}\quad w\in W^L,\ v\in W_L .
\end{equation}

\section{The Leray-Hirsch isomorphism} \label{sec:lhi}

In this section and the next, the equivariant oriented cohomology theory,
$\BBh_\star(\cdot)$, is assumed to be Chern complete over the point for
$T$. We begin by reviewing the algebraic description of the $T$-equivariant
oriented cohomology theory of partial flag varieties developed by Calm\`es,
Lenart, Zainoulline, Zhong, and others, then we define the
algebraic Leray-Hirsch homomorphism, and finally we state the main theorem.

\subsection{Formal group algebras and the $T$-equivariant oriented cohomology
  of a point \cite[\S3] {calmeszainoullinezhong:equivariant} \cite[\S2]
  {calmespetrovzainoulline:invariants} }\label{ssec:S}

The study of $T$-equivariant oriented cohomology begins with the algebraic
description of $R=\BBh_T(\pt)$. Let $F$ be the associated formal group law and
let $R[[\Lambda]]$ be the ring of formal power series in a set of variables
$x_\lambda$, indexed by $\Lambda$ and considered as a topological ring with
the adic topology defined by the ideal generated by $\{\, x_\lambda\mid
\lambda\in \Lambda\,\}$. Define the \emph{formal group algebra of $\Lambda$}
to be the quotient
\[
  R[[\Lambda]]_F= R[[\Lambda]]/\CJ_F,
\]
where $\CJ_F$ is the closure of the ideal generated by $\{x_0\} \cup \{\,
x_{\lambda+\mu}-F(x_\lambda, x_\mu)\mid \lambda, \mu\in \Lambda\,\}$. Denote
the image of $x_\lambda$ in $R[[\Lambda]]_F$ also by $x_\lambda$.

Because $\BBh_\star(\cdot)$ is Chern complete, the rule that maps $x_\lambda$
to the first equivariant Chern class of the line bundle on $G/B$ such that $T$
acts on the fibre over $B$ by the character $\lambda$ defines an isomorphism
$S\xrightarrow{\,\sim\,} R[[\Lambda]]_F$ (see \cite[Theorem
3.3]{calmeszainoullinezhong:equivariant}).  From now on we identify $S$ with
$R[[\Lambda]]_F$ using this isomorphism.

The $W$-action on $\Lambda$ induces an $R$-linear action of $W$-action on $S$
with $z\cdot x_\lambda= x_{z(\lambda)}$ for $z\in W$ and $\lambda\in \Lambda$.

\subsection{}\label{ssec:S+}

It is important to know that $S$ has the structure of a ring of formal power
series and to identify the augmentation ideal.

Let $S_+$ be the ideal in $S$ generated by $\{\,x_\lambda \mid \lambda\in
\Lambda\,\}$. It is shown in \cite[Cor.~2.13]
{calmespetrovzainoulline:invariants} that if $\{\omega_1,\dots ,\omega_n\}$ is
a basis of $\Lambda$, then there is an isomorphism of topological $R$-algebras
\[
  \tau\colon S= R[[\Lambda]]_F\xrightarrow{\,\sim\,} R[[x_{\omega_1}, \dots
  ,x_{\omega_n}]]\quad\text{with}\quad \tau(x_{\omega_i})= x_{\omega_i}
  \text{ for $1\leq i\leq n$.} 
\]
Thus, $\tau$ carries $S_+$ isomorphically onto $R[[x_{\omega_1}, \dots
,x_{\omega_n}]]_+$. Therefore, every element $q\in S$ can be written uniquely
as $q=r\cdot 1+q_+$, where $r\in R$ and $q_+\in S_+$. Moreover, $q$ is a unit
in $S$ if and only if $r$ is a unit in $R$.

\subsection{}\label{ssec:ualpha}

Suppose $\lambda\in \Lambda$. Using the defining relations in $S$ and that $F$
is a formal group law (see \cref{ssec:oct}\cref{eq:21}), we have
\[
  0= x_{\lambda-\lambda} =F(x_\lambda, x_{-\lambda}) =x_\lambda+ x_{-\lambda}
  - x_\lambda x_{-\lambda} G(x_\lambda, x_{-\lambda}),
\]
so $G(x_{\lambda}, x_{-\lambda}) = x_\lambda\inverse +
x_{-\lambda}\inverse$. To simplify the notation, define
\begin{equation*}
  \label{eq:22}
  \kappa_\lambda = G(x_{\lambda}, x_{-\lambda}) = x_\lambda\inverse +
  x_{-\lambda}\inverse \quad\text{and}\quad u_\lambda = x_\lambda /
  x_{-\lambda} =  -1+\kappa_\lambda x_\lambda.
\end{equation*}
Then $\kappa_\lambda$ and $u_\lambda$ are both elements in $S$, and
$u_\lambda$ is a unit in $S$.

\subsection{Formal affine Demazure algebras and their duals
  (\cite{calmeszainoullinezhong:equivariant},
  \cite{calmeszainoullinezhong:coproduct},
  \cite{calmeszainoullinezhong:push})} \label{ssec:alg}

The algebraic description of the $T$-equivariant oriented cohomology of
partial flag varieties is encapsulated in the commutative diagram of
$S$-algebras and $S$-algebra homomorphisms:
\begin{equation}
  \label{eq:fda}
  \vcenter{ \vbox{  \xymatrix{ \BBh_T(G/B) \ar[d]_{\Theta}^{\cong}
        \ar[r]^-{\pi_*} & \BBh_T(G/P) \ar[d]_{\Theta_P}^{\cong}
        \ar[r]^-{\pi^*} & \BBh_T(G/B)  \ar[d]_{\Theta}^{\cong} \\
        \bfD^* \ar@{->>}[r]^-{Y_P\bullet(\cdot)} \ar@{^(->}[d] &
        (\bfD^*)^{W_L} \ar@{^(->}[d]  \ar@{^(->}[r]& \bfD^* \ar@{^(->}[d]\\
        Q_W^* \ar@{->>}[r]^-{Y_P\bullet(\cdot)} & (Q_W^*)^{W_L} \ar@{^(->}[r]
        & Q_W^* ,}}} 
\end{equation}
where the notation (adapted from \cite{kostantkumar:equivariant} and
\cite{calmeszainoullinezhong:equivariant}) is as follows.
\begin{enumerate}
\item $Q$ is the localization of $S$ at $\{\, x_\beta\mid \beta\in \Phi
  \,\}$. Define $x_L$ in $S$ (and $Q$) by 
  \[
    x_L= \prod_{\beta\in \Phi_L^+} x_{-\beta}.
  \]
  For example, if $P=G$, then $x_G= \prod_{\beta\in \Phi^+}
  x_{-\beta}$. \label{it:alg1}
  
\item $Q_W$, the \emph{twisted group algebra,} is defined to be the free
  $Q$-module $Q\rtimes R[W]$ with basis $\{\, \delta_z\mid z\in W\,\}$. $Q_W$
  is also an $R$-algebra with multiplication determined by the rule
  \[
    (q\delta_z)\cdot (q'\delta_{z'})=qz(q')\delta_{zz'}, \quad\text{for
      $q,q'\in Q$ and $z, z'\in W$.}
  \]

  Define $Y_P\in Q_W$ by
  \[
    Y_P= \sum_{v\in W_L} \delta_v x_L\inverse.
  \]
  If $s=s_\alpha$ is a simple reflection define $Y_s=Y_{P_s}$, so
  \[
    Y_s= (\delta_e +\delta_s) x_{-\alpha}\inverse = x_{-\alpha}\inverse
    \delta_e +x_{\alpha}\inverse \delta_s.
  \]
  If $I=(s_1, \dots, s_p)$ is a sequence of simple reflections in $W$ define
  $Y_I=Y_{s_1} \dotsm Y_{s_p}$. For $z\in W$ set $Y_z=Y_{I_z}$. Notice that
  $Y_{wv} =Y_wY_v$ if $w\in W^L$ and $v\in W_L$. It is not hard to check that
  $\{\, Y_z\mid z\in W\,\}$ is a $Q$-basis of $Q_W$.

  Taking $P=P_s$ in diagram \cref{eq:fda} one sees that the elements $Y_s$
  encode the push-pull operators in equivariant oriented cohomology. These
  elements are called \emph{push-pull elements.} In another direction, the
  Demazure operators on the coordinate ring of the Lie algebra of $T$ are
  encoded by the \emph{Demazure elements}, which are defined by
  \[
    X_s= x_\alpha\inverse (\delta_e - \delta_s)= x_\alpha\inverse \delta_e
    -x_\alpha\inverse \delta_s
  \]
  for a simple reflection $s=s_\alpha \in W$. If $I=(s_1, \dots, s_p)$ is a
  sequence of simple reflections in $W$, define $X_I=X_{s_1} \dotsm X_{s_p}$,
  and for $z\in W$ set $X_z=X_{I_z}$. As with the $Y$'s, $X_{wv} =X_wX_v$ if
  $w\in W^L$ and $v\in W_L$, and $\{\, X_z\mid z\in W\,\}$ is a $Q$-basis of
  $Q_W$. \label{it:alg2}

\item $Q_W^* = \Hom_Q(Q_W,Q)$ is the $Q$-dual of $Q_W$. Let $\{\, Y_z^*\mid
  z\in W\,\}$ and $\{\, X_z^*\mid z\in W\,\}$ be the bases of $Q_W^*$ dual to
  the bases $\{\, Y_z\mid z\in W\,\}$ and $\{\, X_z\mid z\in W\,\}$ of $Q_W$,
  respectively.

  The $\bullet$-action of $Q_W$ on $Q_W^*$ is defined by
  \[
    (h\bullet f) (h')=f(h'h), \quad h, h'\in Q_W, f\in Q_W^*.
  \]
  The map $Y_P\bullet(\cdot)$ in diagram \cref{eq:fda} denotes the left
  $\bullet$-action of $Y_P$ on $Q_W^*$, namely $f\mapsto Y_P\bullet f$ for
  $f\in Q_W^*$. Restricting the $\bullet$-action to the basis
  $\{\,\delta_z\mid z\in W\,\}$ of $Q_W$ defines an action of $W$ on $Q_W^*$
  by $Q$-algebra automorphisms. The space of $W_L$-invariants for this action
  is denoted by $(Q_W^*)^{W_L}$. By
  \cite[Lem.~6.5]{calmeszainoullinezhong:push}, $Y_P\bullet Q_W^* =
  (Q_W^*)^{W_L}$. \label{it:alg3}

\item The \emph{formal affine Demazure algebra of the based root system
    determined by $T\subseteq B\subseteq G$,} denoted by $\bfD$, is the
  $R$-subalgebra of $Q_W$ generated by $S$ and $\{\, X_s \mid \text{$s\in W$
    is a simple reflection}\,\}$. Since $\delta_e=1$ and $\delta_s= 1-x_\alpha
  X_s$, it is easy to see that $\{\, \delta_z\mid z\in W\,\} \subseteq \bfD$,
  and using the results in \cite[\S6]{HMSZ:formal} it is straightforward to
  check that $\bfD$ is a free $S$-module with bases $\{\, X_z\mid z\in W\,\}$
  and $\{\, Y_z\mid z\in W\,\}$. \label{it:alg4}

\item $\bfD^* = \Hom_S (\bfD, S)$ is the $S$-dual of $\bfD$. We consider
  $\bfD^*$ as a commutative $S$-algebra with the product induced from the
  coproduct on $\bfD$ defined in \cite[\S8, \S11]
  {calmeszainoullinezhong:coproduct}.  Since $\bfD$ is a free $S$-submodule of
  $Q_W$ that contains the $Q$-basis $\{\,\delta_z\mid z\in W\,\}$ of $Q_W$, it
  follows that every $S$-module homomorphism $\bfD \to S$ extends uniquely to
  a $Q$-module homomorphism $Q_W\to Q$. We identify $\bfD^*$ with an
  $S$-submodule of $Q_W^*$ using this correspondence. It is not hard to see
  that $\bfD^*$ is a $W$-stable subset of $Q_W^*$ (with respect to the
  $\bullet$-action of $W$). Moreover, by \cite[Lem.~11.7]
  {calmeszainoullinezhong:push}, $Y_P\bullet \bfD^* =
  (\bfD^*)^{W_L}$. \label{it:alg5}

\item The inclusion of $T$-fixed points $(G/B)^T\hookrightarrow G/B$ induces
  an $S$-algebra homomorphism $\BBh_T(G/B) \to \BBh_T( (G/B)^T)$. Because
  $(G/B)^T=\{\, zB\mid z\in W\,\}$, we see that $\BBh_T ((G/B)^T)$ is
  isomorphic to $\bigoplus_{z\in W} \BBh_T(zB)$, and so may be identified with
  the $S$-algebra of functions $W\to S$, with pointwise operations, and hence
  with $S_W^*= \Hom_S(S_W,S)$. This leads to an $S$-algebra homomorphism
  $\BBh_T(G/B) \to S_W^*$. Using the basis $\{\,\delta_z\mid z\in W\,\}$ of
  $Q_W$ we identify $Q_W^*$ with the $Q$-algebra of functions $W\to Q$ and get
  an embedding of $S$-algebras, $S_W^* \hookrightarrow Q_W^*$.

  Composing the maps in the preceding paragraph gives an $S$-algebra
  homomorphism $\BBh_T(G/B) \to Q_W^*$. By
  \cite[\S10]{calmeszainoullinezhong:push} this mapping is injective with
  image equal to $\bfD^*$. The isomorphism $\Theta$ in diagram \cref{eq:fda}
  is obtained from the embedding $\BBh_T(G/B) \hookrightarrow Q_W^*$ by
  restricting the codomain to $\bfD^*$. The $S$-algebra homomorphism
  $\Theta_P$ is defined similarly.

\item Recall that $I\rev$ denote the reverse of sequence $I$ and let $f_e\in
  Q_W^*$ denote the $Q$-linear homomorphism with $f_e(\delta_e)=1$ and
  $f_e(\delta_z)=0$ if $z\ne e$. Then by
  \cite[Lem.~8.8]{calmeszainoullinezhong:equivariant},
  \[
    \Theta(\xi_z) = Y_{I_z\rev} \bullet x_G f_e \quad\text{and}\quad
    \Theta_P(\xi^P_w) = Y_PY_{I_w\rev} \bullet x_G f_e
  \]
  for $z\in W$ and $w\in W^L$. Thus, $\{\, Y_{I_z\rev}\bullet x_G f_e\mid z\in
  W\,\}$ is an $S$-basis of $\bfD^*$ and $\{\, Y_P Y_{I_w\rev}\bullet x_G
  f_e\mid w\in W^L\,\}$ is an $S$-basis of $(\bfD^*)^{W_L}$.

  Finally, by \cite[Lem.~14.3] {calmeszainoullinezhong:push} and
  \cite[Lem.~11.6] {calmeszainoullinezhong:push}, the sets $\{\, Y_P
  X_{I_w\rev}\bullet x_G f_e\mid w\in W^L\,\}$ and $\{\, X_w^*\mid w\in
  W^L\,\}$ are $S$-bases of $(\bfD^*)^{W_L}$. On the other hand, if $w\in
  W^L$, then in general $Y_w^*\notin (\bfD^*)^{W_L}$.
\end{enumerate}

\subsection{Duality \cite{calmeszainoullinezhong:push}} \label{ssec:algdual1}

An application of the diagram \cref{ssec:alg} \cref{eq:fda} is the algebraic
description of the scalar pairings on $\BBh_T(G/B)$ and $\BBh_T(G/P)$. By
transport of structure via the isomorphisms $\Theta$ and $\Theta_P$, the
non-degenerate $S$-bilinear forms $\langle\, \cdot\,, \, \cdot\,
\rangle_{G/B}$ and $\langle\, \cdot\,, \, \cdot\, \rangle_{G/P}$ determine
non-degenerate $S$-bilinear forms $\langle\, \cdot\,, \, \cdot\,
\rangle_{\bfD^*}$ and $\langle\, \cdot\,, \, \cdot\, \rangle_{(\bfD^*)^{W_L}}$
on $\bfD^*$ and $(\bfD^*)^{W_L}$, respectively. 

Because $G/G = \pt$, $W_G=W$, and $(\bfD^*)^{W} \cong S$, the mapping
\[
  (a_{G/B})_* \colon \BBh_T(G/B) \to \BBh_T(\pt)=S
\]
is given algebraically by the mapping
\[
  Y_G\bullet(\cdot) \colon \bfD^*\to (\bfD^*)^{W} \cong S.
\]
It follows that 
\[
  \langle f,f'\,\rangle_{\bfD^*} =Y_G\bullet (ff') \quad\text{for}\quad
  f,f'\in \bfD^*.
\]
By \cite[Thm.~15.5] {calmeszainoullinezhong:push}, $\{\, Y_z^*\mid z\in W\,\}$
is the $S$-basis of $\bfD^*$ that is dual to the basis $\{\,
Y_{I_z\rev}\bullet x_G f_e\mid z\in W\,\}$. Therefore
\begin{equation}
  \label{eq:20}
  \Theta(\xi^z)= Y_z^*\quad \text{for $z\in W$.}
\end{equation}

For $G/P$ and $(\bfD^*)^{W_L}$, the projection $a_{G/B}\colon G/B\to G/G=\pt$
factors as $a_{G/B}= a_{G/P} \circ \pi$, the element $Y_G\in Q_W$ factors as
$Y_{G,L}Y_P$, where $Y_{G,L}= \sum_{w\in W^L} \delta_w x_G\inverse x_L$, and
there is a commutative diagram
\begin{equation*}
  \vcenter{\vbox{  \xymatrix{\BBh_T(G/B) \ar[d]_{\Theta}^{\cong}
        \ar@{->>}[r]^-{\pi_*} & 
        \BBh_T(G/P) \ar[d]_{\Theta_P}^{\cong} \ar@{->>}[r]^-{a_*} &
        \BBh_T(\pt)  \ar[d]_{\Theta_G}^{\cong}\\ 
        \bfD^* \ar@{->>}[r]^-{Y_P\bullet(\cdot)} & (\bfD^*)^{W_L}
        \ar@{->>}[r]^-{Y_{G,L} \bullet(\cdot)}& \bfD^* .} }}
\end{equation*}
It follows that
\[
  \langle f,f'\,\rangle_{(\bfD^*)^{W_L}} =Y_{G,L}\bullet (ff')
  \quad\text{for}\quad f,f'\in (\bfD^*)^{W_L}.
\]

By \cite[Thm.~15.6] {calmeszainoullinezhong:push}, the bases $\{\, Y_P
X_{I_w\rev}\bullet x_G f_e\mid w\in W^L\,\}$ and $\{\, X_w^*\mid w\in W^L\,\}$
are dual bases of $(\bfD^*)^{W_L}$.

Define
\begin{equation}
  \label{eq:8}
  Z_w^* = \Theta_P( \xi_P^w) = \Theta(\zeta^w) \quad \text{for $w\in W^L$.}
\end{equation}
Then $\{\, Z_w^*\mid w\in W^L\,\}$ is the $S$-basis of $(\bfD^* )^{W_L}$ dual
to the basis $\{\, Y_P Y_{I_w\rev}\bullet x_G f_e\,\}$. 

\subsection{} \label{ssec:algdual2}

By the projection formula in equivariant oriented cohomology, the mappings
$\pi_*$ and $\pi^*$ are adjoint with respect to the scalar products on
$\BBh_T(G/B)$ and $\BBh_T(G/P)$. It then follows that $\pi^* \pi_* \colon
\BBh_T(G/B) \to \BBh_T(G/B)$ is self-adjoint.

By transport of structure via $\Theta$ and $\Theta_P$ in
\cref{ssec:alg}\cref{eq:fda} one sees that
\begin{itemize}
\item the mapping $Y_P\bullet(\cdot) \colon \bfD^* \to (\bfD^*)^{W_L}$ and the
  inclusion $(\bfD^*)^{W_L} \hookrightarrow \bfD^*$ are adjoint:
  \[
    \langle \tilde f, f\rangle_{\bfD^*} = \langle \tilde f, Y_P\bullet f
    \rangle_{(\bfD^*)^{W_L}} \quad\text{for $\tilde f\in
      (\bfD^*)^{W_L}$ and $f\in \bfD^*$.}
  \]
  and
\item the mapping $Y_P\bullet(\cdot) \colon \bfD^* \to \bfD^*$ is
  self-adjoint:
  \[
    \langle Y_P\bullet f, f' \rangle_{\bfD^*}= \langle f, Y_P\bullet
    f'\rangle_{\bfD^*} \quad\text{for $f,f'\in \bfD^*$}.
  \]
\end{itemize}

\subsection{Executive Summary}

To help streamline the notation and save bullets, for $z\in W$ define
$Y_z^{\times}\in \bfD^*$ by
\[
  Y_z^{\times}= Y_{I_z\rev} \bullet x_Gf_e \quad\text{and}\quad X_z^{\times}=
  X_{I_z\rev} \bullet x_Gf_e.
\]

\begin{itemize}
\item The $S$-module $\BBh_T(G/B)$ has Bott-Samelson and dual Bott-Samelson
  bases, $\{\, \xi_z\mid z\in W\,\}$ and $\{\, \xi^z\mid z\in W\,\}$,
  respectively; the mapping $\Theta\colon \BBh_T(G/B)\to \bfD^*$ is an
  $S$-algebra isomorphism with
  \[
    \Theta(\xi_z)= Y_z^{\times}, \quad\text{and}\quad \Theta(\xi^z)= Y_z^*;
  \]
  and $\bfD^*$ has pairs of dual bases
  \begin{itemize}
  \item $\{\, Y_{z}^\times\mid z\in W\,\}$ and $\{\, Y_z^*\mid
    z\in W\,\}$, and
  \item $\{\, X_{z}^\times\mid z\in W\,\}$ and $\{\, X_z^*\mid z\in W\,\}$.
  \end{itemize}

\item The description of $\BBh_T(G/P)$ is similar, but less complete. The
  $S$-module $\BBh_T(G/P)$ has Bott-Samelson and dual Bott-Samelson bases,
  $\{\, \xi^P_w\mid w\in W^L\,\}$ and $\{\, \xi_P^w\mid w\in W^L\,\}$,
  respectively; the mapping $\Theta_P\colon \BBh_T(G/P)\to (\bfD^*)^{W_L}$
  is an $S$-algebra isomorphism with
  \[
    \Theta_P(\xi^P_w)= Y_P \bullet Y_w^{\times}, \quad\text{and}\quad
    \Theta_P(\xi_P^w)= Z_w^*;
  \]
  and $(\bfD^*)^{W_L}$ has pairs of dual bases
  \begin{itemize}
  \item $\{\, Y_P \bullet Y_w^{\times}\mid w\in W^L\,\}$ and $\{\, Z_w^*\mid
    w\in W^L\,\}$, and
  \item $\{\, Y_P\bullet X_{w}^\times\mid w\in W^L\,\}$ and $\{\, X_w^*\mid
    w\in W^L\,\}$.
  \end{itemize}
  The expansion of $Z_w^*$ in the $Y_z^*$-basis is not well understood in
  general. Some partial information is given in \cref{cor:zw}.

\item The algebraic descriptions of $\BBh_T(G/B)$ and $\BBh_T(G/P)$ are
  compatible by the commutativity of the upper left square in
  \cref{ssec:alg}\cref{eq:fda}: $\big(Y_P\bullet (\cdot) \big) \circ \Theta =
  \Theta_P \circ \pi_*$.
\end{itemize}

Notice that it follows from the preceding constructions that the $R$-algebra
$\BBh_T(\pt)$ and the $\BBh_T(\pt)$-algebra structure of $\BBh_T(G/P)$, depend
only on the ring $R$, the formal group law $F$, and the root system $\Phi$,
but not on the oriented cohomology theory $\BBh$.

\subsection{The formal affine Demazure algebra $\bfD_L$ and
  $\mathbf{\BBh_T(P/B)}$} \label{ssec:DL}

Given $R$, note that $S$ depends only on $T$, whereas $Q$, $\bfD$, and
$\bfD^*$ depend on the based root system determined by the inclusions
$T\subseteq B \subseteq G$.

Let $\bfD_L$ be the formal affine Demazure algebra of the based root system
determined by $T\subseteq B\cap L \subseteq L$. We identify $\bfD_L$ with the
$S$-subalgebra of $\bfD$ that has $S$-bases $\{\, Y_v\mid v\in W_L\,\}$ and
$\{\, X_v\mid v\in W_L\,\}$. When $Y_v$, respectively $X_v$, is considered as
an element of $\bfD_L$ it is denoted by $Y_{v,L}$, respectively $X_{v,L}$.

Let $i_a\colon \bfD_L\to \bfD$ be the inclusion. Then $i_a(Y_{v,L})= Y_v$ and
$i_a(X_{v,L})= X_v$. Hence, taking $S$-duals, $\bfD_L^*$ may be identified with
the $S$ submodule of $\bfD^*$ with bases $\{\, Y_v^*\mid v\in W_L\,\}$ and
$\{\, X_v^*\mid v\in W_L\,\}$, and the dual map
$i_a^*\colon \bfD^* \to \bfD_L^*$ is the projection given by
\[
  i_a^*(Y_z^*)=\begin{cases} Y_{v, L}^* &\text{if $z=v\in W_L$,} \\
    0 &\text{if $z\notin W_L$,} \end{cases} \quad\text{and}\quad
  i_a^*(X_z^*)=\begin{cases} X_{v, L}^* &\text{if $z=v\in W_L$,} \\
    0 &\text{if $z\notin W_L$.} \end{cases}
\]

It is not difficult to check that there is a commutative diagram
\begin{equation}
  \label{eq:dl}
  \vcenter{\vbox{  \xymatrix{\BBh_T(G/B) \ar[d]_\Theta^{\cong}
        \ar[r]^-{i^*} & \BBh_T(P/B) \ar[d]_{\Theta_L}^{\cong} \\ 
        \bfD^* \ar[r]^-{i_a^*} & \bfD_L^* ,} }}
\end{equation}
where $\Theta_L \colon \BBh_T(P/B)\to \bfD_L^*$ is the analog of $\Theta$ for
the flag variety $P/B \cong L/(B\cap L)$.

\subsection{The algebraic Leray-Hirsch homomorphism} \label{ssec:alhh}

Define
\[
  j_a\colon \bfD_L^*\to \bfD^* \quad\text{by} \quad j_a( Y_{v, L}^*)= Y_v^*
  \quad\text{for $v\in W_L$},
\]
so $j_a$ is a right inverse of the dual map $i_a^*$. Abusing notation
slightly, also define
\[
  j_a\colon \BBh_T(P/B) \to \BBh_T(G/B) \quad\text{to be the composition
    $\Theta\inverse \circ j_a\circ \Theta_L$.}
\]
Then $j_a$ is a right inverse of the surjection $i^*$.

The \emph{algebraic Leray-Hirsch homomorphism} is the composition
$\varphi_a =\operatorname{mult} \circ (\pi^*\otimes j_a)$. The constructions
in \cref{ssec:alg} and \cref{ssec:DL} piece together to give a commutative
diagram
\begin{equation*}
  \label{eq:alhi}
  \vcenter{\vbox{    \xymatrix{ \BBh_T(G/P) \otimes_{S} \BBh_T(P/B)
        \ar[rr]^-{\pi^* \otimes j_a} \ar@/^30pt/[rrrr]^{\varphi_a}
        \ar[d]_{\Theta_P \otimes \Theta_L}^{\cong} && \BBh_T(G/B) \otimes_{S}
        \BBh_T(G/B) \ar[rr]^-{\operatorname{mult}} \ar[d]_{\Theta\otimes
          \Theta}^{\cong} && \BBh_T(G/B)  \ar[d]_{\Theta}^{\cong} \\
        (\bfD^*)^{W_L} \otimes_{S} \bfD_L^* \ar@{^(->}[rr]^{k\otimes j_a}&&
        \bfD^* \otimes_{S} \bfD^* \ar[rr]^-{\operatorname{mult}} && \bfD^* ,}
    }} 
\end{equation*}
where $k\colon (\bfD^*)^{W_L} \to D^*$ is the inclusion. In the $S$-basis
$\{\, X_w^* \otimes X_{v,L}^* \mid w\in W^L, \, v\in W_L\,\}$ of
$(\bfD^*)^{W_L}\otimes \bfD_L^*$, the composition of $\Theta \circ \varphi_a
\circ (\Theta_P\otimes \Theta_L)\inverse$ is given by
\begin{equation}
  \label{eq:18}
  \Theta \circ \varphi_a \circ (\Theta_P\otimes \Theta_L)\inverse (X_w^*
  \otimes X_{v,L}^* )= X_w^* j_a(X_{v,L}^*) = X_w^* X_v^*.  
\end{equation}
It follows from \cref{ssec:glhh}\cref{eq:19}, \cref{ssec:algdual1}\cref{eq:8},
and \cref{ssec:algdual1}\cref{eq:20} that
\[
  \Theta \circ \varphi_g \circ (\Theta_P\otimes \Theta_L)\inverse (Z_w^*
  \otimes Y_{v,L}^* )= Z_w^* j_a(Y_{v,L}^*) = Z_w^* Y_v^* .
\]
Thus, $\varphi_a \ne \varphi_g$. As observed above, no geometric
interpretation of the classes $\Theta\inverse(X_z^*)$ is known.

We can now state the main theorem.

\begin{theorem}
  \label{thm:main}
  Suppose that $T\subseteq B\subseteq P\subseteq G$ are complex algebraic
  groups as in \cref{ssec:gp} and that $\BBh_\star(\cdot)$ is an equivariant
  oriented cohomology theory that is Chern complete over the point for
  $T$. Then the Leray-Hirsch homomorphisms $\varphi_g$ and $\varphi_a$ are
  $S$-module isomorphisms.
\end{theorem}

\section{Proof of \cref{thm:main}} \label{sec:pf}

In this section we continue to assume that $\BBh_\star(\cdot)$ is Chern
complete over the point for $T$ and prove \cref{thm:main}, namely that
$\varphi_g$ is an isomorphism and that $\varphi_a$ is an isomorphism. Both
proofs follow the same basic argument outlined in the next subsection, but the
details are more delicate for the geometric Leray-Hirsch homomorphism. We give
the full proof in this case and leave much of the simpler case of the
algebraic Leray-Hirsch homomorphism to the reader.

\subsection{Outline of the proof}\label{ssec:ol}

Considering the geometric Leray-Hirsch homomorphism, define $\psi_g=
\Theta\circ \varphi_g \circ (\Theta_P \otimes \Theta_L)\inverse$, so the
diagram
\[
  \xymatrix{ \BBh_T(G/P) \otimes_{S} \BBh_T(P/B) \ar[rr]^-{\varphi_g}
    \ar[d]_{\Theta_P \otimes \Theta_L}^{\cong} && \BBh_T(G/B)
    \ar[d]_{\Theta}^{\cong} \\ 
    (\bfD^*)^{W_L} \otimes_{S} \bfD_L^* \ar[rr]^-{\psi_g}&& \bfD^* }
\]
commutes. Therefore, to prove that $\varphi_g$ is an isomorphism, it is
sufficient to show that $\psi_g$ is an isomorphism. Because the domain and
codomain are free $S$-modules with the same rank, to show that $\psi_g$ is an
isomorphism it is enough to show that it is surjective. It follows from
\cref{ssec:glhh}\cref{eq:19}, \cref{ssec:algdual1}\cref{eq:20}, and
\cref{ssec:algdual1}\cref{eq:8} that
\[
  \psi_g(Z_w^* \otimes Y_{L,v}^*) = Z_w^* Y_v^*, 
\]
and so to show that $\psi_g$ is surjective, it is enough to show that $\{\,
Z_w^* Y_v^*\mid w\in W^L,\, v\in W_L\,\}$ spans $\bfD^*$.

Let $C$ denote the $S$-valued, $|W| \times |W|$-matrix with entries
$c_{w,v}^{w'',v''}$, where
\[
  Z_w^* Y_v^*= \sum_{w''\in W^L} \sum_{v''\in W_L} c_{w,v}^{w'',v''}
  Y_{w''v''}^*.
\]
To show that $\{\, Z_w^* Y_v^*\mid w\in W^L,\, v\in W_L\,\}$ spans $\bfD^*$,
it is enough to show that $C$ is invertible. In turn, the matrix $C$ is
invertible if its determinant is a unit in $S$. Using the identification of
$S$ with a ring of formal power series given by the isomorphism $\tau$ in
\cref{ssec:S+}, it suffices to show (1) that $C$ is upper triangular mod
$S_+$, in which case every entry in the usual expansion of $\det C$ as a sum
of products will lie in $S_+$, with the possible exception of the product of
the diagonal entries, and (2) that the diagonal entries all lie in
$1+S_+$. Therefore, $\det C \in 1+S_+$, and so is a unit in $S$.  After some
preliminary results, assertions (1) and (2) are proved in \cref{pro:u} and
\cref{pro:d}, respectively. An example of the matrix $C$ is given in
\cref{ssec:exa1}.

\subsection{Example -- type $A_2$} \label{ssec:exa1}

Suppose $G=\operatorname{SL}_3(\BBC)$, with simple reflections $s=s_\alpha$
and $t=s_\beta$, and that $P=P_s$. Since the rank of $W$ is equal to $2$, the
only element in $W$ that does not have a unique reduced expression is
$sts=tst$. We take $I_{sts} =(s,t,s)$. The linear order on $W$ described in
\cref{ssec:gp} is
\[
  e \prec s\prec t\prec ts \prec st \prec sts.
\]

Recall the elements $u_\gamma= x_\gamma x_{-\gamma}\inverse$ (in $1+S_+$)
defined in \cref{ssec:ualpha}. The coefficients in the expansion of
$Z_w^*Y_v^*$ in the basis $\{\, Y_z^*\mid z\in W\,\}$, and hence the entries
of the matrix $C$, are given in \cref{tab:1}. In the table, the entries marked
$*$ do not contribute to $\det C$.
\begin{table}[htb!]
  \caption{Entries of the matrix $C$}
  \centering \renewcommand\arraystretch{1.3}
  \begin{tabular}{>{$}c<{$}||>{$}c<{$} >{$}c<{$}|>{$}c<{$} >{$}c<{$}
    |>{$}c<{$} >{$}c<{$}}
    &Y^*_{e}&Y^*_{s}&Y^*_{t}&Y^*_{ts}&Y^*_{st} &Y^*_{w_0} \\
    \hline\hline 
    Z^*_{e} Y^*_e&1&0&*&*&*&* \\
    Z^*_{e} Y^*_{s}&0&1&*&*&*&*\\
    \hline 
    Z^*_{t} Y^*_{e} &0&0&-u_\beta&0&*&*\\ 
    Z^*_{t} Y^*_{s}&0&0&0&-u_\beta&*&*\\  \hline  
    Z^*_{st} Y^*_{e}&0&0&0&0&u_\alpha u_{\alpha+\beta}&0  \\ 
    Z^*_{st} Y^*_{s}&0
            &0&0&0&-x_\alpha u_{\alpha+\beta} &- u_{\alpha+\beta}
  \end{tabular}
  \label{tab:1}
\end{table}

By \cref{pro:d}, $c_{w,v}^{w,v}\in u+S^+$, where $u$ is a unit in $S$
determined by $w$ and $v$. For example, $c^{t,s}_{t,s}\in u_\beta
u_{\alpha+\beta} +S_+ $. This is consistent with the diagonal entries in the
table, but it is not obvious:
\[
  u_\beta u_{\alpha+\beta} \kappa_{\alpha+\beta}- u_\beta x_{\alpha+\beta} =
  -u_\beta, \quad \text{and so} \quad u_\beta u_{\alpha+\beta} +S_+ = -u_\beta
  +S_+.
\]

Also, in this example, each diagonal block of $C$ is lower triangular and is
in fact a diagonal matrix modulo $S_+$. Examples in type $A_3$ show that in
general the diagonal blocks of $C$ are neither lower triangular, nor diagonal,
modulo $S_+$. Rather, it is shown in \cref{pro:u} that the diagonal blocks of
$C$ are upper triangular modulo $S_+$.

\subsection{Base change matrices} \label{ssec:bc}

We begin with a lemma that quantifies the non-linearity, in $S$, of the
mapping that carries $Y_I$ to $Y_I^\times = Y_{I\rev} \bullet x_G f_e$, for a
sequence, $I$, of simple reflections.

It follows from the definitions that when $Y_z$ is expanded in terms of the
basis $\{\,\delta_z\mid z\in W\,\}$ of $Q_W$, the coefficient of $\delta_y$ is
equal to zero unless $y\leq z$, and that the coefficient of $\delta_z$ is a
unit in $Q$.  For $y,z\in W$ define $a_{z,y}\in Q$ and $b_{z,y}\in S$ by
\begin{equation}
  \label{eq:25}
  Y_z= \sum_{y\leq z} a_{z,y} \delta_y \quad \text{and}\quad \delta_z=
  \sum_{y\leq z} b_{z,y} Y_y ,  
\end{equation}
so the matrices $(a_{z,y})_{z,y\in W}$ and $(b_{z,y})_{z,y\in W}$ are
inverses. Let $\{\, f_z\mid z\in W\,\}$ be the basis of $Q_W^*$ that is dual
to the basis $\{\, \delta_z\mid z\in W\,\}$. By duality we have
\[
  Y_x^*= \sum_{z\in W} b_{z,x} f_z \quad \text{and} \quad f_x= \sum_{z\in W}
  a_{z,x} Y_z^*.
\]
Using the $\bullet$ action it is shown in \cite[Lem.~7.3]
{calmeszainoullinezhong:push} that
\begin{equation}
  \label{eq:14}
  Y_z^{\times}= \sum_{y\leq z} a_{z,y} y(x_G) f_y,\quad\text{and so}
  \quad f_z=  \sum_{y\leq z} z(x_G)\inverse b_{z,y} \cdot Y_y^{\times}.
\end{equation}

\begin{lemma}\label{lem:multqs}
  Suppose $q\in Q$ and $z\in W$. Then
  \[
    qY_{I_z\rev} \bullet x_G f_e = \sum_{\substack{w\\ w\leq z}} \Big(
    \sum_{\substack{y\\ w\leq y \leq z}} y(q) \cdot a_{z,y} \cdot b_{y,w}
    \Big) \,Y_w^{\times} .
  \]
\end{lemma}

\begin{proof}
  It is shown in \cite[Cor.~7.2] {calmeszainoullinezhong:push} that
  $Y_{I_z\rev} = \sum_{y, y\inverse \leq z} y(a_{z,y\inverse} x_G\inverse)
  \cdot x_G \cdot \delta_{y}$ and it is straightforward to check that
  $(p\delta_y)\bullet (qf_z)= q\cdot zy\inverse(p) \cdot f_{zy\inverse}$ for
  $p,q\in Q$ and $y,z\in W$. Thus,
  \[
    qY_{I_z\rev} \bullet x_G f_e = \sum_{\substack{y\\ y\inverse \leq z}}
    \big(q \cdot y(a_{z,y\inverse} x_G\inverse) \cdot x_G \cdot \delta_{y}
    \big)\bullet x_G f_e= \sum_{\substack{y\\ y \leq z}} \big(y(q x_G) \cdot
    a_{z,y} \big) f_{y} .
  \]
  Then using \cref{ssec:bc}\cref{eq:14} to replace $f_y$ by $\sum_{w\leq y}
  y(x_G)\inverse b_{y,w} Y_w^\times$ and simplifying gives the result.
\end{proof}

\subsection{}\label{ssec:eww}

The next few lemmas lead to a description of the expansion of $Z_w^*$ with
respect to the basis $\{\, Y_z^*\mid z\in W\,\}$ of $\bfD^*$. First we need
some definitions.

\begin{enumerate}
\item If $w,w'\in W^L$ and $v\in W_L$, then $Y_P\bullet Y_{wv}^{\times} \in
  (\bfD^*)^{W_L}$ and we define $e_{wv, w'}\in S$ by \label{eq:10}
  \begin{equation*}
    Y_P\bullet Y_{wv}^{\times}= \sum_{w'\in W^L} e_{wv, w'} \big(Y_P\bullet
    Y_{w'}^{\times} \big). 
  \end{equation*}

\item The action of $W$ on $Q$ extends linearly to an action of $Q_W$ on $Q$,
  which is denoted by ``$\cdot$'':
  \[
    \big(\sum_{z\in W} q_z\delta_z\big)\cdot q' = \sum_{z\in W} q_z\,z(q').
  \]
  
\item For a simple reflection $s=s_\alpha$ define
  \[
    \Delta_s\colon Q\to Q\quad\text{by}\quad \Delta_s(q)= \frac {q-s(q)}
    {x_\alpha} .
  \]
  It is shown in \cite[Corollary 3.4]{calmespetrovzainoulline:invariants} that
  $\Delta_s(S) \subseteq S$.
\end{enumerate}

\begin{lemma}\label{lem:yPyI}
  Suppose that $I= (s_1, \dots, s_p)$ with $s_1, \dots, s_p \in W_L$. Then
  \[
    Y_P Y_{I\rev}= Y_P\big(Y_I\cdot 1\big).
  \]
\end{lemma}

\begin{proof}
  The proof is by induction on $p$. Suppose $s=s_\alpha$ is a simple
  reflection in $W$. If $I=\emptyset$, then $Y_PY_I=Y_P=Y_P (Y_\emptyset \cdot
  1)$, and if $I=(s)$, then $Y_P Y_s= Y_P \kappa_\alpha= Y_P (Y_s\cdot 1)$,
  where $\kappa_\alpha=x_{-\alpha}\inverse +x_{\alpha}\inverse$ (see
  \cref{ssec:ualpha}).

  More generally, an easy computation shows that if $q\in Q$, then
  \[
    qY_s= Y_ss(q) + \Delta_s(q) \quad \text{and}\quad \kappa_\alpha s(q)
    +\Delta_s(q) = Y_s \cdot q .
  \]
  Now by induction,
  \begin{align*}
    Y_P Y_{I\rev}%
    &= Y_P \big(Y_{s_2}\dotsm Y_{s_p}\cdot 1 \big) Y_{s_1} \\
    &= Y_P \Big(Y_{s_1} s_1\big( Y_{s_2}\dotsm Y_{s_p}\cdot 1 \big)
      +\Delta_{s_1} \big(Y_{s_2}\dotsm Y_{s_p}\cdot 1 \big) \Big)\\
    &= Y_P \Big(\kappa_{\alpha_1} s_1\big( Y_{s_2}\dotsm Y_{s_p}\cdot 1 \big)
      +\Delta_{s_1} \big(Y_{s_2}\dotsm Y_{s_p}\cdot 1 \big) \Big)
      = Y_P \big(Y_I\cdot 1\big).
  \end{align*}
\end{proof}

\begin{lemma}\label{lem:ypywv}
  Suppose $w\in W^L$ and $v\in W_L$. Then
  \[
    e_{wv,w'}=0 \text{ unless $w\geq w'$} \quad\text{and}\quad e_{wv, w}=
    w(Y_v\cdot 1).
  \]
  Thus,
  \[
    Y_P\bullet Y_{wv}^{\times} = w(Y_v\cdot 1) \big(Y_P \bullet Y_w^\times
    \big)+ \sum_{\substack{w'\in W^L\\ w'< w}} e_{wv, w'} \big(Y_P\bullet
    Y_{w'}^{\times} \big).
  \]
\end{lemma}

\begin{proof}
  The conclusions of the lemma follow from the definition of $Y_{wv}^\times$,
  \cref{lem:yPyI}, and \cref{lem:multqs}:
  \begin{align*}
    Y_P\bullet Y_{wv}^\times%
    & =Y_P\bullet (Y_{v}\cdot1) Y_{I_w\rev} \bullet x_G f_e \\
    &= w(Y_v\cdot 1) \big(Y_P \bullet Y_w^\times \big)+ \sum_{\substack{z\in
    W\\z<w}} \Big( \sum_{\substack{y\in W z\leq y\leq w}} y(Y_v\cdot 1)
    a_{w,y}b_{y,z} \Big) \big(Y_P \bullet Y_z^\times \big) .
  \end{align*}
\end{proof}

The next corollary follows formally from \cref{lem:ypywv}, the fact that
$Y_P\bullet(\cdot)$ is adjoint to the inclusion of $(\bfD^*)^{W_L}$ in
$\bfD^*$, and duality, because $\{\, Z_w^* \mid w\in W^L\,\}$ and $\{\,
Y_P\bullet Y_w^\times \mid w\in W^L\,\}$ are dual bases of $(\bfD^*)^{W_L}$.

\begin{corollary}\label{cor:zw}
  Suppose $w\in W^L$. Then
  \[
    Z_w^* = \sum_{v\in W_L} w(Y_v\cdot 1) Y_{wv}^* + \sum_{\substack{w'v'\in
        W\\ w'>w}} e_{w'v', w} Y_{w'v'}^*.
  \]
\end{corollary}

\subsection{Structure constants and the matrix $C$} \label{ssec:scme}

Define structure constants $p_{u,v}^w\in S$ by
\[
  Y_u^* Y_v^* = \sum_w p_{u,v}^w Y_w^*.
\]
An explicit formula for $p_{u,v}^w$, to be described below, is given in
\cite[Thm.~4.1] {goldinzhong:structure}. One consequence of the formula is
that if $p_{u,v}^w \ne 0$, then $u\leq w$ and $v\leq w$.

Using \cref{ssec:eww}\cref{eq:10} and duality we can write
\[
  Z_w^*Y_v^*= \Big(\sum_{w'\in W^L} \sum_{v'\in W_L} e_{w'v', w} \,Y_{w'v'}^*
  \Big) Y_v^* = \sum_{w''\in W^L} \sum_{v''\in W_L} \Big( \sum_{w'\in W^L}
  \sum_{v'\in W_L} e_{w'v', w}\, p_{w'v', v}^{w''v''} \Big) Y_{w''v''}^*.
\]
Define $c^{w'',v''}_{w,v}$ to be the coefficient of $Y_{w''v''}^*$ in $Z_w^*
Y_v^*$, so
\[
  c^{w'',v''}_{w,v} = \sum_{w'\in W^L} \sum_{v'\in W_L} e_{w'v', w}\, p_{w'v',
    v}^{w''v''}.
\]

Using the linear order on $W$ from \cref{ssec:gp}, define $C$ be the $|W|
\times |W|$ matrix whose $(wv, w''v'')$ entry is
$c^{w'',v''}_{w,v}$. Similarly, for $w, w''\in W^L$, let $C^{w,w''}$ be the
$|W_L| \times |W_L|$ matrix whose $(v, v'')$ entry is $c^{w'',v''}_{w,v}$. By
definition, $C$ is a block matrix with blocks $C^{w,w''}$ for $w, w''\in
W^L$. By \cref{lem:ypywv}, $e_{w'v', w}=0$ unless $w'\geq w$, and as noted
above, $p_{w'v',v}^{w''v''} =0$ unless $w''v''\geq w'v'$, so $e_{w'v', w}\,
p_{w'v', v}^{w''v''}=0$ unless $w''\geq w'\geq w$. Therefore,
$c^{w'',v''}_{w,v}=0$ unless $w''\geq w$, and so $C^{w,w''}=0$ unless $w''
\geq w$. Thus, $C$ is an upper triangular block matrix with diagonal
blocks $C^{w,w}$. It follows from \cref{cor:zw} that the $(v, v'')$-entry of
$C^{w,w}$ is
\begin{equation*}
  \label{eq:16}
  c^{w,v''}_{w,v} = \sum_{v'\in W_L} w(Y_{v'}\cdot 1)\, p_{wv', v}^{wv''} .  
\end{equation*}

\subsection{Explicit formulas for $p^{wv''}_{wv',v}$ and
  $c^{w,v''}_{w,v}$} \label{ssec:sc}

Suppose $w\in W^L$ and $v,v',v''\in W_L$.

Let $I_w=(s_1, \dots, s_p)$ and let $I_{v''}=(s_{p+1}, \dots, s_{p+q})$. For
$1\leq j\leq p+q$ set $s_j=s_{\alpha_j}$. By assumption, $I_{wv''}= I_w \amalg
I_{v''}$ is the concatenation of $I_w$ and $I_{v''}$.
\begin{itemize}
\item For a sequence, $I$, of simple reflections and $z\in W$, define
  $b_{I,z}$ to be the coefficient of $Y_z$ in the expansion of $Y_I$, so
  \[
    Y_I=\sum_{z\in W} b_{I,z}Y_z.
  \]
  
\item For subsequences $E$ and $F$ of $I_{wv''}$ and $1\leq j\leq p+q$, define
  $B_j^{E,F}\in \bfD$ by
  \begin{equation}
    \label{eq:11}
    B_{j}^{E,F}=  \begin{cases}
      x_{\alpha_j} \delta_{s_j}&\text{if $s_j\in E\sqcap F$,}\\
      -u_{\alpha_j}\delta_{s_j}&\text{if $s_j\in (E\setminus F) \sqcup
                                 (F\setminus E)$,}\\ 
      x_{-\alpha_j} \inverse + \big(x_{\alpha_j} x_{-\alpha_j}^{-2} \big)
      \delta_{s_j}&\text{if $s_j\notin E\sqcup F$.}
    \end{cases}
  \end{equation}

\end{itemize}
With this notation, Goldin and Zhong \cite[Thm.~4.1] {goldinzhong:structure}
prove that
\begin{equation}
  \label{eq:gz}
  p_{wv',v}^{wv''} = \sum_{E,F \sqsubseteq I_{wv''}} \big(B_{1}^{E,F} \dotsm
  B_{p+q}^{E,F} \cdot 1 \big) b_{E,wv'} b_{F,v}.
\end{equation}
Therefore,
\begin{equation}
  \label{eq:41}
  c^{w,v''}_{w,v} = \sum_{v'\in W_L} \sum_{E,F \sqsubseteq I_{wv''}}
  w\big( Y_{v'}\cdot 1 \big)\big( B^{E,F}_1 \dotsm B^{E,F}_{p+q} \cdot 1
  \big) b_{E,wv'} b_{F, v}.
\end{equation}

\subsection{}

Recall that $W$ is the group generated by the simple reflections subject to
the braid relations and the relations $s^2=e$ for each simple reflection
$s$. Let $\Wtilde$ be the semigroup generated by the simple reflections in $W$
subject to the braid relations and the relations $s^2=s$ for each simple
reflection $s$. As sets $W=\Wtilde$, but obviously the semigroup operation,
say $*$, in $\Wtilde$ is not equal the group operation in $W$. For a sequence
$I=(s_1, \dots, s_p)$ of simple reflections in $W$ define
\[
  \wtilde(I) =s_1*\dotsm *s_p\in \Wtilde
\]
and consider $\wtilde(I)$ as an element in $W$. For example, if $s_1\ne s_2$
are simple reflections, then $\wtilde(s_1, s_2, s_2)= s_1s_2$, and if $I$ is a
reduced expression of $z\in W$, then $\wtilde(I) =z$. In general, there is a
subsequence $(s_{i_1}, \dots, s_{i_k}) \sqsubseteq I$ such that
$\wtilde(I)=s_{i_1} \dotsm s_{i_k}$.

\begin{lemma}\label{lem:ulev}
  Suppose $y, z\in W$ and $I$ is a sequence of simple reflections.
  \begin{enumerate}
  \item If $b_{I,y} \ne 0$, then $y\leq \wtilde(I)$. \label{it:u1}
  \item If $I\sqsubseteq I_z$ and $b_{I,y} \ne 0$, then $y\leq
    z$. \label{it:u2}
  \item If $w\in W^L$, $u,v\in W_L$, $J\subseteq I_{wu}$, and $b_{J, wv}
    \ne0$, then $I_w\sqcap J = I_w$. \label{it:u3}
  \end{enumerate}
\end{lemma}

\begin{proof}
  The first assertion is proved in \cite[Lem.~3.2]{goldinzhong:structure}.

  The second assertion follows by first writing $Y_I$ as a $Q$-linear
  combination of $\{\, \delta_u\mid u\in W\,\}$, then writing each $\delta_u$
  as a linear combination of $\{\, Y_t\mid t\in W\,\}$, and then observing
  that
  \begin{itemize}
  \item if $\delta_u$ appears in the expansion of $Y_I$, then $u\leq z$, by
    the subword property of the Bruhat order, and
  \item if $Y_y$ appears in the expansion of $\delta_u$, then $y\leq u$, by
    \cref{ssec:bc}\cref{eq:25}.
  \end{itemize}

  With the assumptions of \cref{it:u3}, one has that $J= (I_{w} \sqcap J)
  \sqcup (I_{u} \sqcup J)$, and thus it follows from \cref{it:u1} and the
  definition of $\wtilde$ that $wv \leq \wtilde(J) \leq \wtilde( I_{w} \sqcap
  J) \wtilde (I_{u} \sqcap J)$ and $\wtilde (I_{u} \sqcap J) \in W_L$. Write
  $\wtilde( I_{w} \sqcap J)= w'v'$, where $w'\in W^L$ and $v'\in W_L$. Then
  $wv\leq w'v' \wtilde(I_u\sqcap J)$, so $w\leq w'$. But $w'$ is a subword of
  $w$, so $w'=w$. Now $I_w$ is reduced and $\wtilde( I_{w} \sqcap J)= wv'$,
  hence it must be that $v'=1$ and $I_w\sqcap J=I_w$, as claimed.
\end{proof}

\begin{proposition}\label{pro:u}
  For $w\in W^L$, the matrix $C^{w,w}$ is upper triangular modulo $S_+$.
\end{proposition}

\begin{proof}
  Suppose $w\in W^L$, $v,v''\in W_L$, and consider the formula for
  $c^{w,v''}_{w,v}$ in \cref{ssec:sc}\cref{eq:41}.
  
  We first show that if $v\not \leq v''$, then $c^{w,v''}_{w,v} \in
  S_+$. Indeed, if $v'\in W_L$, $E,F\sqsubseteq I_{wv''}$, and
  \[
    w\big( Y_{v'}\cdot 1 \big) \big( B^{E,F}_1 \dotsm B^{E,F}_{p+q} \cdot 1
    \big) b_{E,wv'} b_{F, v} \ne 0,
  \]
  then $b_{E,wv'} \ne 0$ and $b_{F, v}\ne 0$. It follows from
  \cref{lem:ulev}\cref{it:u3} that $I_w\sqsubseteq E$. By
  \cref{lem:ulev}\cref{it:u1}, $v\leq \wtilde (F)$. If in addition, $I_w
  \sqcap F=\emptyset$, then $F\sqsubseteq I_{v''}$ and so it follows from
  \cref{lem:ulev}\cref{it:u2} that $v\leq v''$. By assumption $v\not \leq
  v''$, so it must be that $I_w \sqcap F\ne \emptyset$, whence $E\sqcap F\neq
  \emptyset$.  But then it follows from \cref{ssec:sc}\cref{eq:11} that $\big(
  B^{E,F}_1 \dotsm B^{E,F}_{p+q} \cdot 1 \big) \in S_+$. Therefore, $\big(
  B^{E,F}_1 \dotsm B^{E,F}_{p+q} \cdot 1 \big) \in S_+$ for all $E$ and $F$,
  and so $c^{w,v''}_{w,v} \in S_+$.

  The contrapositive of the assertion in the preceding paragraph is that if
  $c^{w,v''}_{w,v} \notin S_+$, then $v\leq v''$ and it then follows that
  $v\preceq v''$. Therefore, if $c^{w,v''}_{w,v} \notin S_+$, then $v'' \not
  \prec v$ in the linear order on $W_L$. Equivalently, if $v''\prec v$, then
  $c^{w,v''}_{w,v} \in S_+$, and so every entry of $C^{w,w}$ below the
  diagonal lies in $S_+$.
\end{proof}

\begin{proposition}\label{pro:d}
  Suppose $w\in W^L$ and $v\in W_L$. Then $c^{w,v}_{w,v} \in 1+ S_+$.
\end{proposition}

\begin{proof}
  By \cref{ssec:sc}\cref{eq:41},
  \[
    c^{w,v}_{w,v} = \sum_{v'\in W_L} \sum_{E,F \sqsubseteq I_{wv}} w\big(
    Y_{v'}\cdot 1 \big)\big( B^{E,F}_1 \dotsm B^{E,F}_{p+q} \cdot 1 \big)
    b_{E,wv'} b_{F, v}.
  \]
  Define
  \[
    \gamma_1= \alpha_1, \quad \gamma_2= s_1(\alpha_2), \quad \dots\quad
    \gamma_{p+q}= s_1\dotsm s_{p+q-1} (\alpha_{p+q}).
  \]
  We show that $w\big( Y_{v'}\cdot 1 \big) \big( B^{E,F}_1 \dotsm
  B^{E,F}_{p+q} \cdot 1 \big) b_{E,wv'} b_{F, v} \in S_+$, unless $v'=e$,
  $E=I_w$, and $F=I_v$, and that $w\big( Y_{e}\cdot 1 \big) \big(
  B^{I_w,I_v}_1 \dotsm B^{I_w,I_v}_{p+q} \cdot 1 \big) b_{I_w,w} b_{I_v, v} =
  (-1)^{\ell(wv)} u_{\gamma_1} \dotsm u_{\gamma_{p+q}}$. Consequently,
  \[
    c^{w,v}_{w,v}\in (-1)^{\ell(wv)} u_{\gamma_1} \dotsm u_{\gamma_{p+q}} +S_+
    \subseteq  1+ S_+ ,
  \]
  because each $u_{\gamma_j}\in -1+S_+$.

  To simplify the notation, set
  \[
    S(v',E,F)= w\big( Y_{v'}\cdot 1 \big) \big( B^{E,F}_1 \dotsm B^{E,F}_{p+q}
    \cdot 1 \big) b_{E,wv'} b_{F, v}
  \]
  for $v'\in W_L$ and $E, F \sqsubseteq I_{wv}$.

  Suppose that $S(v',E,F) \ne 0$. By \cref{lem:ulev}\cref{it:u3} we have
  $I_w\sqsubseteq E$.
  \begin{enumerate}
  \item If $F\sqcap I_w \ne \emptyset$, then it follows from
    \cref{ssec:sc}\cref{eq:11} that $B^{E,F}_1 \dotsm B^{E,F}_{p+q} \cdot 1
    \in S_+$, and so $S(v', F, E)\in S_+$. \label{it:1}
  \item If $F\sqcap I_w=\emptyset$, then $F\sqsubseteq I_v$, and so $|F| \leq
    |I_v|$. On the other hand, $b_{F, v} \ne 0$, and so $v\leq \wtilde(F)$ by
    \cref{lem:ulev}\cref{it:u1}. But then $\ell(v)\leq \ell(\wtilde(F)) \leq
    |F|$, and so $|I_v|\leq |F|$. Thus, $F=I_v$. \label{it:2}
  \item \label{it:3} If $F=I_v$ and $E\sqcap I_v \ne \emptyset$, then it
    follows from \cref{ssec:sc}\cref{eq:11} that $B^{E,I_v}_1 \dotsm
    B^{E,I_v}_{p+q} \cdot 1 \in S_+$, and so $S(v',E, I_v)\in S_+$.
  \item Finally, if $F=I_v$ and $E\sqcap I_v =\emptyset$, then $E=I_w$, and
    since $b_{E, wv'}\ne0$, we must also have $v'=e$. \label{it:4}
  \end{enumerate}
  It follows from \cref{it:1}, \cref{it:2}, \cref{it:3}, and \cref{it:4} that
  $S(v',E,F)\in S_+$, unless $F=I_v$, $E=I_w$, and $v'=e$.

  Finally, it follows from \cref{ssec:sc}\cref{eq:11} that $B^{I_w,I_v}_1
  \dotsm B^{I_w,I_v}_{p+q} \cdot 1 = (-1)^{\ell(ww)} u_{\gamma_1} \dotsm
  u_{\gamma_{p+q}}$, and clearly $\big( Y_{e}\cdot 1 \big) b_{I_w,w} b_{I_v,
    v} =1$, so $S(e,I_w,I_v) = (-1)^{\ell(wv)} u_{\gamma_1} \dotsm
  u_{\gamma_{p+q}}$, as claimed.
\end{proof}

\subsection{The algebraic Leray-Hirsch isomorphism}

We use the same general strategy as the preceding proof, but with some
modifications and simplifications.

First, by the argument in \cref{ssec:ol}, but using
\cref{ssec:alhh}\cref{eq:18} in place of \cref{ssec:glhh}\cref{eq:19}, one
sees that it is sufficient to show that the set $\{\, X_w^* X_v^* \mid w\in
W^L, v\in W_L\,\}$ is an $S$-basis of $\bfD^*$.

Next, define structure constants $p_{u,v}^w\in S$ by
\[
  X_u^* X_v^* = \sum_{w\in W} p_{u,v}^w X_w^*
\]
and define $C$ to be the $|W| \times |W|$ matrix whose $(wv, w''v'')$ entry is
$p^{w''v''}_{w,v}$.

For $w, w''\in W^L$, $C^{w,w''}$ is the $|W_L| \times |W_L|$ matrix with $(v,
v'')$ entry equal to $p^{w''v''}_{w,v}$. It is shown in \cite[Thm.~4.1]
{goldinzhong:structure} that if $p^{w''v''}_{w,v}\ne 0$, then $w\leq
w''$. Thus, $C$ is an upper triangular block matrix with diagonal blocks
$C^{w,w}$.

Finally, the structure constants $p^{wv''}_{w,v}$ can be computed using the
explicit formula in \cite[Thm.~4.1] {goldinzhong:structure}. Arguments similar
to those in the proofs of \cref{pro:u} and \cref{pro:d} then show that each
$C^{w,w}$ is upper triangular modulo $S_+$ and that the diagonal entries,
$p^{wv}_{w,v}$ of $C^{w,w}$ lie in $1+S_+$. Further details are left to the
reader.

\subsection{Example -- type $A_2$} \label{ssec:exa2}

Continuing the example in \cref{ssec:exa1}, the entries of the matrix $C$ for
the algebraic Leray-Hirsch homomorphism $\varphi_a$ are given in \cref{tab:2}.

\begin{table}[htb!]
  \caption{Entries of the matrix $C$ (for $\varphi_a$)}
  \centering \renewcommand\arraystretch{1.3}
  \begin{tabular}{>{$}c<{$}||>{$}c<{$} >{$}c<{$}|>{$}c<{$} >{$}c<{$}
    |>{$}c<{$} >{$}c<{$}}
    &X^*_{e}&X^*_{s}&X^*_{t}&X^*_{ts}&X^*_{st} &X^*_{w_0} \\
    \hline\hline 
    X^*_{e} X^*_e&1&0&0&0&0&0 \\
    X^*_{e} X^*_{s}&0&1&0&0&0&0\\
    \hline 
    X^*_{t} X^*_{e} &0&0&1&0&0&0\\ 
    X^*_{t} X^*_{s}&0&0&0&1&1&-\kappa_\alpha\\  \hline  
    X^*_{st} X^*_{e}&0&0&0&0&1&0  \\ 
    X^*_{st} X^*_{s}&0&0&0&0&x_\alpha&-u_\alpha
  \end{tabular}
  \label{tab:2}
\end{table}

Again, each diagonal block is lower triangular and is a diagonal matrix modulo
$S_+$, but higher rank examples show that in general the diagonal blocks of
$C$ are neither lower triangular, nor diagonal, modulo $S_+$.

\section{Applications}\label{sec:app}

\subsection{Descent for Leray-Hirsch isomorphisms}

In addition to $\varphi_g$ and $\varphihat_g$, using \cite[Lem.~11.1]
{calmeszainoullinezhong:equivariant} and its proof for the algebraic subgroup
$\langle e\rangle$ of $T$, and the construction in \cref{ssec:glhh}, one can
define a geometric Leray-Hirsch homomorphism
\[
  \varphi_g^{\langle e\rangle}\colon \BBh_{\langle e\rangle}(G/P) \otimes_{R}
  \BBh_{\langle e\rangle}(P/B) \longrightarrow \BBh_{\langle e\rangle}(G/B).
\]
By \cref{thm:main}, $\varphihat_g$ is an $\Shat$-module isomorphism.

\begin{theorem}\label{thm:noet}
  If $R$ is Noetherian, then the geometric Leray-Hirsch homomorphisms
  $\varphi_g^{\langle e\rangle}$ and $\varphi_g$ are $R$-module and $S$-module
  isomorphisms, respectively. In particular, the geometric Leray-Hirsch
  homomorphisms are isomorphisms for Chow rings, equivariant Chow rings,
  $K$-theory, and equivariant $K$-theory.
\end{theorem}

\begin{proof}
  It follows from \cref{thm:bssp} that $\varphi_g$ is obtained from
  $\varphi_g^{\langle e\rangle}$ by applying the base change functor
  $S\otimes_R-$, and that $\varphihat_g$ is obtained from $\varphi_g^{\langle
    e\rangle}$ by applying the composition of base change functors
  $(\Shat\otimes_S-) \circ (S\otimes_R-) =(\Shat\otimes_R-)$.

  Recall from \cref{ssec:S+} that $\Shat\cong R[[x_{\omega_1}, \dots
  ,x_{\omega_n}]]$ is a formal power series ring. Since $R$ is Noetherian,
  $\Shat$ is a faithfully flat $R$-algebra. Thus $\varphihat_g$ is an
  $\Shat$-module isomorphism, so $\varphi_g^{\langle e\rangle}$ is an
  $R$-module isomorphism by ``faithfully flat descent.''  It then follows that
  $\varphi_g$ is an $S$-module isomorphism by the functoriality of base
  change.
\end{proof}

\subsection{}
In the rest of this section $\varphi\colon \BBh_T(G/P) \otimes_{\BBh_T}
\BBh_T(P/B) \to \BBh_T(G/B)$ denotes an arbitrary Leray-Hirsch homomorphism,
not necessarily $\varphi_g$ or $\varphi_a$.

\subsection{Leray-Hirsch isomorphisms in (non-equivariant) oriented
  cohomology}\label{ssec:neq}

Let $H$ be a subgroup of $T$. Applying the base change functor $\BBh_H
\otimes_{\BBh_T} (\cdot)$ to $\varphi$, and composing with the natural
isomorphism
\begin{multline*}
  \BBh_H\otimes_{\BBh_T} \Big( \BBh_T(G/P) \otimes_{\BBh_T}
  \BBh_T(P/B)\Big) \\ \cong \Big(\BBh_H\otimes_{\BBh_T} \BBh_T(G/P)
  \Big) \otimes_{\BBh_H} \Big(\BBh_H\otimes_{\BBh_T}
  \BBh_T(P/B) \Big) ,
\end{multline*}
we obtain an $\BBh_H$-algebra homomorphism
\begin{equation*}
  \label{eq:1}
  \Big(\BBh_H\otimes_{\BBh_T} \BBh_T(G/P) \Big) \otimes_{\BBh_H}
  \Big(\BBh_H\otimes_{\BBh_T} \BBh_T(P/B) \Big) \rightarrow
  \BBh_H\otimes_{\BBh_T} \BBh_T(G/B).  
\end{equation*}
By \cite[Lem.~11.1] {calmeszainoullinezhong:equivariant} the natural
homomorphism
\[
  \BBh_H\otimes_{\BBh_T}\BBh_T(G/P) \to \BBh_H(G/P)
\]
induced by the restriction functor $\BBh_T(\cdot) \to \BBh_H(\cdot)$ is an
isomorphism. If $j$ is a section of $i^*$, then $j$ determines a section of
$i^*\colon \BBh_H \otimes_{\BBh_T} \BBh_T(P/B)\to \BBh_H \otimes_{\BBh_T}
\BBh_T(G/B)$. Therefore, the base change functor $\BBh_H \otimes_{\BBh_T}
(\cdot)$ applied to the Leray-Hirsch homomorphism $\varphi$ (in
$T$-equivariant cohomology) induces a Leray-Hirsch homomorphism
\[
  \varphi_H\colon \BBh_H(G/P) \otimes_{\BBh_H} \BBh_H(P/B) \xrightarrow{}
  \BBh_H(G/B)
\]
in $H$-equivariant cohomology, and if $\varphi$ is an isomorphism, then so is
$\varphi_H$. Taking $H=\langle e\rangle$ gives following proposition.

\begin{proposition}
  With the hypotheses in \cref{thm:main}, the Leray-Hirsch isomorphisms
  $\varphi_g$ and $\varphi_a$ induce geometric and algebraic Leray-Hirsch
  isomorphisms in the oriented cohomology theory $\BBh(\cdot)$ associated to
  $\BBh_\star(\cdot)$.
\end{proposition}

\subsection{}

Recall from \cref{ssec:fglex} that the universal oriented cohomology theory,
denoted by $\Omega^*(\cdot)$, is algebraic cobordism. Extending scalars from
$\BBZ$ to $\BBQ$ one obtains the oriented cohomology theory
$\Omega^*(\cdot)_{\BBQ}$.

Heller and Malag\'on-L\'opez \cite[Pro.~7]{hellermalagonlopez:equivariant}
prove a Leray-Hirsch theorem in $\Omega^*(\cdot)_{\BBQ}$ for a projective
locally isotrivial fibration of smooth varieties, $p\colon X\to Y$, when the
fibre, $F$, is smooth, projective, and paved by affine spaces ($F$ is called a
``cellular variety'' in \cite {hellermalagonlopez:equivariant}). As with the
Leray-Hirsch Theorem in singular cohomology, \cite[Pro.~7]
{hellermalagonlopez:equivariant} prescribe certain sections of $i^*\colon
\Omega^*(X)_{\BBQ} \to \Omega^*(F)_{\BBQ}$ such that the resulting
Leray-Hirsch homomorphism is an isomorphism. It is also shown that if $p$ is
Zariski-locally trivial, then the Leray-Hirsch homomorphism in
$\Omega^*(\cdot)$ is an isomorphism.

Now suppose $\Omega_G^*(\cdot)$ is an equivariant oriented cohomology theory
such that associated non-equivariant theory is $\Omega^*(\cdot)$. Then
\cite[Pro.~7]{hellermalagonlopez:equivariant} applies to the fibration
$\pi\colon G/B\to G/P$ and gives a family of sections of $i^*$ such that the
resulting Leray-Hirsch homomorphisms are isomorphisms. Since algebraic
cobordism is Chern complete over the point for $T$, \cref{thm:main} provides
two explicit sections of $i^*$ that satisfy the hypotheses in
\cite[Pro.~7]{hellermalagonlopez:equivariant}.

\subsection{Modules and characters}

The projection $\pi\colon G\to G/P$ induces an $S$-algebra homomorphism
$\pi^*\colon \BBh_T(G/P)\to \BBh_T(G/B)$, which may be identified with the
inclusion of $S$-algebras $(\bfD^*)^{W_L} \subseteq \bfD^*$. Thus, the
following corollary is an immediate consequence of \cref{thm:main}.

\begin{corollary}
  The ring $\bfD^*$ is a free $(\bfD^*)^{W_L}$-module with basis $\{\,
  X_v^*\mid v\in W_L \,\}$ and the ring $\BBh_T(G/B)$ is a free
  $\BBh_T(G/P)$-module with basis $\{\, \xi_L^v \mid v\in W_L\, \}$.
\end{corollary}

\subsection{}

Now recall the $\bullet$-action of $W$ on $\bfD^*$ from
\cref{ssec:alg}\cref{it:alg3} and \cref{ssec:alg}\cref{it:alg5}. By transport
of structure via the isomorphism $\Theta$, the $W$ action on $\bfD^*$
determines an action of $W$ on $\BBh_T(G/B)$ by $S$-algebra automorphisms.

If $y,z\in W$, then $y(f_z)= \delta_y \bullet f_z= f_{zy\inverse}$. It follows
that the $\bullet$-action of $W$ on $Q_W^*$ induces the regular
representation.

Because $W$ acts on $\bfD^*$ as $S$-algebra automorphisms, if $w\in W^L$ and
$v, v'\in W_L$, then $X_w^*\in (\bfD^*)^{W_L}$ and so $v'( X_w^* X_v^*) =
X_w^* \cdot v'(X_v^*)$. Thus, the following corollary, which generalizes
\cite[Thm.~4.7]{drellichtymoczko:module}, is an immediate consequence of
\cref{thm:main}.

\begin{corollary}
  Let $\chi_L$ be the character of the $\bullet$-action of $W_L$ on
  $\BBh_T(G/B)$, and let $\chi$ be the character of the $\bullet$-action of
  $W_L$ on $\BBh_T(P/B)$, then $\chi_L=|W^L|\, \chi$.
\end{corollary}

\subsection{Alternate formulations: The Borel model and the $\bullet$ and
  $\odot$ $W$-actions}

As a final application, we give an alternate formulation of the Leray-Hirsch
isomorphisms using the so-called Borel model of $\BBh_T(G/B)$ and the
$\bullet$- and $\odot$-actions of $W$ on $\BBh_T(G/B)$.

In equivariant $K$-theory there are two maps from $K_T(\pt)$, which may be
identified with the representation ring of $T$, to $K_T(G/B)$. One is induced
from the canonical map $a_{G/B}\colon G/B \to \pt$ and the other, called the
\emph{characteristic map}, and denoted by $c$, is given by the rule that maps
a character $\lambda$ of $T$ to the isomorphism class of the $T$-equivariant
line bundle on $G/B$ on which $T$ acts on the fibre over $B$ as
$\lambda$. These two maps induce an isomorphism $K_T(\pt) \otimes_{K_T(\pt)^W}
K_T(\pt) \xrightarrow{\,\cong\,} K_T(G/B)$ given by $p\otimes q\mapsto
a_{G/B}^*(p) \cdot c(q)$

From now on, in addition to the standing assumptions in \cref{ssec:asm} we
also assume that the torsion primes of $\Phi$ are invertible in $R$.

A generalization of the characteristic map,
\[
  \ch_g\colon \BBh_T \to \BBh_T(G/B),
\]
is defined in \cite[\S10] {calmeszainoullinezhong:equivariant} for any
equivariant oriented cohomology theory and also called the
\emph{characteristic map}. Set
\[
  \ch_a= \Theta\circ \ch_g, \quad\text{so}\quad \ch_a\colon S\to \bfD^*.
\]
It is shown in \cite[Lem.~10.1] {calmeszainoullinezhong:equivariant} that the
composition $\ch_a(p)= p\bullet 1$. One checks that $\ch_a(p)= \sum_{w\in W}
w(p) f_w$ and that $\ch_a$ is $W$-equivariant with respect to the natural
action of $W$ on $S$ and the $\bullet$-action of $W$ on $\bfD^*$. It is shown
in \cite[Thm.~11.4] {calmeszainoullinezhong:coproduct} that the map
\[
  \rho\colon S\otimes_{S^W}S \to \bfD^*\quad \text{with} \quad \rho(p\otimes
  q)= p \cdot\ch_a(q)
\]
is an $S$-algebra isomorphism, where $S$-module structure on $S\otimes_{S^W}S$
is given by the action of $S$ on the left-hand factor. This is the \emph{Borel
  model} of $\bfD^*$.

\subsection{}\label{ssec:eq2}
The group $W$ acts independently on each factor of the $S$-algebra
$S\otimes_{S^W}S$. It is shown in \cite[Lem~3.7]
{lenartzainoullinezhong:parabolic} that the isomorphism $\rho$ intertwines the
action of $W$ on the right-hand factor with the $\bullet$-action of $W$ on
$\bfD^*$, and the action of $W$ on the left-hand factor with an action of $W$
on $\bfD^*$, which is denoted by $\odot$. In other words, for $z\in W$, the
diagrams
\[
  \vcenter{\vbox{\xymatrix{S\otimes _{S^W}S\ar[r]^-\rho \ar[d]_-{\id\otimes z}
        & \bfD^*\ar[d]^-{z\bullet(\cdot)}\\
        S\otimes_{S^W}S\ar[r]^-\rho & \bfD^*}}} \quad\text{and}\quad
  \vcenter{\vbox{\xymatrix{S\otimes _{S^W}S\ar[r]^-\rho \ar[d]_-{z\otimes \id}
        & \bfD^*\ar[d]^-{z\odot (\cdot)}\\
        S\otimes_{S^W}S\ar[r]^-\rho & \bfD^*}}}
\]
commute. By transport of structure via the isomorphism $\Theta$, the
$\bullet$- and $\odot$-actions of $W$ on $\bfD^*$ define commuting actions of
$W$ on $\BBh_T(G/B)$, also denoted by $\bullet$ and $\odot$, respectively.  It
is shown in \cite[\S9] {calmeszainoullinezhong:equivariant} that the
$\bullet$-action of $W$ on $\BBh_T(G/B)$ coincides with the action arising
from the right action of $W$ on $G/T$. Roughly speaking, the $\odot$-action of
$W$ on $\BBh_T(G/B)$ arises from the action of $W$ on the Picard group of
$G/B$.

In order to distinguish the $W$-actions, let $(\bfD^*)^{(W_L, \bullet)}$ and
$(\bfD^*)^{(W_L, \odot)}$ denote the $W_L$-invariants with respect to the
$\bullet$-action and the $\odot$-action, respectively. Similar notation will
be used for $\BBh_T(G/B)$. For example,
\begin{equation}
  \label{eq:2}
  \BBh_T(G/P) \cong (\bfD^*)^{(W_L, \bullet)} \cong \BBh_T(G/B)^{(W_L,
    \bullet)} ,
\end{equation}
and a Leray-Hirsch homomorphism is a map
\[
  \varphi\colon \BBh_T(G/B)^{(W_L, \bullet)} \otimes_{\BBh_T} \BBh_T(P/B)\to
  \BBh_T(G/B).
\]
We can use the $\odot$ action to make the domain of $\varphi$ symmetric with
respect to taking $W_L$-invariants.

\subsection{}\label{ssec:eq4}
Replacing $W$ by $W_L$, there is a characteristic map $\ch_a^L\colon S\to
\bfD^*_L$ and an $S$-algebra isomorphism $\rho^L\colon S\otimes_{S^{W_L}}S \to
\bfD^*_L$.

Consider the chain of isomorphisms
\[
  S\cong S^{W_L}\otimes _{S^{W_L}}S\cong (S\otimes_{S^{W_L}}S)^{(W_L,
    \odot)}\xrightarrow[\cong]{\rho^L} (\bfD^*_L)^{(W_L, \odot)} .
\]
It is straightforward to check that the composition is given by the
characteristic map for $W_L$, namely $p\mapsto \rho^L(1\otimes p)=\ch_a^L(p)
\in \bfD^*_L$. Thus, there are isomorphisms
\begin{equation}
  \label{eq:4}
  \BBh_T(P/B)^{(W_L, \odot)} \cong (\bfD^*_L)^{(W_L, \odot)} \cong S .
\end{equation}

\begin{corollary}
  If $j\colon \BBh_T(P/B) \to \BBh_T(G/B)$ is a right inverse to $i^*\colon
  \BBh_T(G/B) \to \BBh_T(P/B)$ and $\varphi$ is the resulting Leray-Hirsch
  homomorphism, then $\varphi$ may be identified with the homomorphism
  \[
    \BBh_T(G/B)^{(W_L, \bullet)}\otimes_{\BBh_T^{W_L}} \BBh_T(P/B)^{(W_L,
      \odot)} \to \BBh_T(G/B)
  \]
  given by the composition of the multiplication map in $\BBh_T(G/B)$ with
  $\id\otimes j$.

  In particular, the Leray-Hirsch isomorphisms $\varphi_g$ and $\varphi_a$
  each induce an isomorphism
  \[
    \BBh_T(G/B)^{(W_L, \bullet)}\otimes_{\BBh_T^{W_L}} \BBh_T(P/B)^{(W_L,
      \odot)} \xrightarrow{\, \cong\,} \BBh_T(G/B).
  \]
\end{corollary}

\begin{proof}
  It is straightforward to check that if $k\colon (D^*_L)^{(W_L, \odot)} \to
  D^*_L$ denotes the inclusion, then
  \[
    \id\otimes k\colon (D^*)^{(W_L, \bullet)} \otimes_{S^{W_L}} (D^*_L)^{(W_L,
      \odot)} \to (D^*)^{(W_L, \bullet)} \otimes_S D_L^*
  \]
  is an isomorphism and the diagram
  \begin{equation}
    \label{eq:5}
    \vcenter{\vbox{ \xymatrix{(D^*)^{(W_L, \bullet)} \otimes_{S^{W_L}}
          (D^*_L)^{(W_L, \odot)}  \ar[rr]^-{\id\otimes k}_-{\cong}
          \ar[dr]_{\mult \circ (\id \otimes j)} 
          && (D^*)^{W_L} \otimes_{S} D^*_L \ar[dl]^{\varphi} \\
          &D^*&}}}
  \end{equation}
  commutes. The proof follows from \cref{eq:5}, using
  \cref{ssec:eq2}\cref{eq:2}, \cref{ssec:eq4}\cref{eq:4}, and the maps
  $\Theta$, $\Theta_L$, and $\Theta_P$.
\end{proof}

\appendix
\section{Equivariant $K$-theory}\label{sec:apdx}

In this appendix we give a direct construction of the geometric Leray-Hirsch
homomorphism in $T$-equivariant $K$-theory for the fibre bundle $G/B \to G/P$,
and sketch a proof that it is an isomorphism.

Most of the notation in this appendix is either as defined above, but tailored
for the special case of equivariant $K$-theory, or taken from
\cite[\S2]{grahamkumar:positivity}. Note that $K_T(\cdot)$ not Chern complete
over the point for $T$ because $K_T(\pt)$ is the representation ring of $T$.

\subsection{} \label{ssec:a1}

First, for $(H,X)\in \GrSmC$, $K_G(X)$ denotes the Grothendieck group of the
category of $H$-equivariant coherent sheaves of $\CO_X$-modules. By the
equivariant version of a theorem of Borel and Serre, every class in $K_G(X)$
contains an $H$-equivariant vector bundles on $X$. The constructions in
\cite[Ch.~5] {chrissginzburg:representation} show that $K_\star( \, \cdot\,)$
is an equivariant oriented cohomology theory. The associated oriented
cohomology theory, $K_{\langle e\rangle}(\cdot)$, is algebraic $K$-theory, and
the formal group law is the multiplicative formal group law $F(t,u)= t+u-tu$.

\subsection{} \label{ssec:a2}

In $K$-theory, $R=K(\pt)$ is canonically isomorphic to $\BBZ$, and
$S=K_T(\pt)$ is canonically isomorphic to representation ring of $T$. Because
$T$ is a torus, we may identify $S$ with the group algebra
$\BBZ[\Lambda]$. Each character, $\lambda\in \Lambda$, indexes an element
$e^\lambda\in S$, and the multiplication in $S$ is determined by $e^\lambda
e^\mu=e^{\lambda+\mu}$. The augmentation ideal, denoted by $S_+$, is the
kernel of the ring homomorphism $S\to \BBZ$ that maps each $e^\lambda$ to
$1$. Let $\widehat S$ denote the $S_+$-adic completion of $S$.

With the notation in \cref{ssec:S}, let $\BBZ[\Lambda]_F$ denote the image of
the subring of $\BBZ[[\Lambda]]_F$ generated by $\{\, x_\lambda\mid \lambda\in
\Lambda\,\}$. The isomorphism $\widehat S \xrightarrow{\, \sim\,}
\BBZ[[\Lambda]]_F$ maps $1-e^{-\lambda}$ to $x_\lambda$ and restricts to an
isomorphism $S \xrightarrow{\, \sim\,} \BBZ[\Lambda]_F$ (see \cite[Example
2.20]{calmespetrovzainoulline:invariants}). We identify these two rings using
this isomorphism. For example, we use the equalities
\[
  x_\lambda= 1-e^{-\lambda}, \quad x_{-\lambda}= x_\lambda / (x_\lambda-1) =
  1-e^\lambda, \quad\text{and}\quad u_\lambda= x_\lambda/ x_{-\lambda}=
  x_\lambda-1= -e^{-\lambda}
\]
in arguments below.

\subsection{} \label{ssec:a3}

Suppose $w\in W^L$. If $I=I_w$, then the Bott-Samelson map $q_{I}^P\colon
Z_{I}^P \to G/P$ in \cref{ssec:bs} factors through a rational resolution
$Z_{I}^P \to \overline{B\cdot wP}$. Therefore, applying the functor
$K_T(\cdot)$ to $q_{I}^P$, we see that
\[
  \xi^P_w= (q_{I}^P)_*( 1) = (q_{I}^P)_*\big( [\CO_{Z_I^P}] \big) =
  [\CO_{\overline{ B\cdot wP}}]
\]
is the class of the structure sheaf of $\overline{ B\cdot wP}$, and hence that
$\xi^P_w$ does not depend on the choice of $I_w$.

The constructions in \cref{sec:lhh} apply verbatim to define a geometric
Leray-Hirsch homomorphism,
\begin{equation}
  \label{eq:glhia}
  \vcenter{\vbox{    \xymatrix{ K_T(G/P) \otimes_{S} K_T(P/B)
        \ar[rr]^-{\pi^* \otimes j_g} \ar@/^30pt/[rrrr]^{\varphi}
        && K_T(G/B) \otimes_{S}  K_T(G/B) \ar[rr]^-{\operatorname{mult}}
        && K_T(G/B), } }}  
\end{equation}
where
\begin{equation*}
  \label{eq:19a}
  j_g( \xi_L^v)= \xi^v\quad\text{and} \quad \varphi(\xi_P^w \otimes \xi_L^v)=
  \pi^*(\xi_P^w) j_g( \xi_L^v) =\zeta^w \xi^v .
\end{equation*}
To show that $\varphi$ is an $S$-module isomorphism, we follow the same
overall strategy as in \cref{ssec:ol}, but with modifications made possible
because the Bott-Samelson classes do not depend on the choice of a
resolution. These modifications lead to more precise results.

\begin{itemize}
\item The expansion of $\zeta^w$ in the basis $\{\,\xi^z\mid z\in W\,\}$ of
  $K_T(G/B)$ is well-understood \cite[Lemma 3.4]{grahamkumar:positivity}:
  $\zeta^w= \sum_{v\in W_L} \xi^{wv}$.

\item The entries of the matrix $C$ are $c^{w'',v''}_{w,v}= \sum_{v'\in W_L}
  p^{w''v''}_{wv',v}$, where $p^{w''v''}_{wv',v}$ is the coefficient of
  $\xi^{w''v''}$ in $\xi^{wv'} \xi^v$.
  
\item The structure constants, $p^z_{x,y}$, can be computed either using
  results of Kostant and Kumar, or the formula in
  \cref{ssec:sc}\cref{eq:gz}. In particular, if $p^{w''v''}_{wv',v} \ne0$,
  then $w''\geq w$, so the matrix $C$ is block upper triangular with diagonal
  blocks $C^{w,w}$ for $w\in W^L$.
\item The matrix $C^{w,w}$ is lower triangular with units on the
  diagonal.  
\end{itemize}
This last statement is consistent with what has already been proved: it is
easy to see that modulo $S_+$, the matrix $C^{w,w}$ is diagonal, and hence
upper triangular.

To prove that the Leray-Hirsch homomorphism in \cref{eq:glhia} is an
isomorphism, it remains to show that $C^{w,w}$ is lower triangular with units
on the diagonal. This is accomplished in \cref{lem:coeff0} and \cref{cor:lam}.

In the following, $\star$ denotes a sum of terms $r_z\xi^z$ with $r_z\in S$
and $z\in W_{>w}^LW_L$, where $W_{>w}^L=\{\, w'\in W^L\mid w'>w\,\}$. For
example,
\[
  \zeta^w \xi^v=\sum_{v''} c^{w,v''}_{w,v} \xi^{wv''} + \star,
\]
because $C$ is block upper triangular.

\begin{lemma}\label{lem:coeff0}
  If $c_{w,v}^{w,v''}\ne0$ then $v''\leq v$. 
\end{lemma}

\begin{proof}
  In this proof we use results for push-pull operators proved in
  \cite{kostantkumar:equivariant}. In order to minimize the notational
  overhead, we use the notation developed in \cref{sec:lhi}.

  For a simple reflection $s=s_\alpha$ $\pi_s\colon G/B \to G/P_s$ be the
  projection. Define $A_s=\pi_s^* (\pi_{s})_* \colon K_T(G/B)\to
  K_T(G/B)$. One can check that as in diagram \cref{ssec:alg}\cref{eq:fda},
  $\Theta\circ A_s= \big(Y_s \bullet(\,\cdot\,) \big) \circ \Theta$. Computing
  in $\bfD^*$, it is straightforward to check that
  \[
    A_s(\xi^z)=\begin{cases} \xi^{zs}+\xi^z &\text{if $zs<z$,} \\
      0&\text{if $z<zs$}
    \end{cases} \quad\text{and that}\quad A_s(f)A_s(g)= A_s\big( A_s(f)g \big),
  \]
  for $z\in W$ and $f,g\in K_T(G/B)$. It follows from the first equality that
  $A_s(\zeta^w)= \zeta^w$ for $w\in W^L$.
  
  The proof of the lemma proceeds by induction on $\ell(w_L)- \ell(v)$, where
  $w_L$ is the longest element in $W_L$. The result is clearly true for
  $v=w_L$. Suppose $v\in W_L$, $vs<v$, and the conclusion of the lemma holds
  for $v$. Then
  \[
    A_s(\zeta^w \xi^v)= A_s(A_s(\zeta^w) \xi^v)= A_s(\zeta^w) A_s(\xi^v)=
    \zeta^w \xi^{vs}+\zeta^w\xi^v.
  \]
  On the other hand, by induction $\zeta^w \xi^{v} = \sum_{v''\leq v}
  c_{w,v}^{w,v''} \xi^{wv''} + \star$. Apply $A_s$ to both sides to get
  \begin{align*}
    A_s(\zeta^w \xi^{v}) &= \sum_{v''\leq v} c_{w,v}^{w,v''} A_s(\xi^{wv''})
    +\star = \sum_{\substack{v''\leq v\\ v''s<v''}}
    c_{w,v}^{w,v''}(\xi^{wv''s} +\xi^{wv''}) +\star \\
    &=  \sum_{\substack{v''\leq vs\\ v''<v''s}} c_{w,v}^{w,v''s}\xi^{wv''} +
    \sum_{\substack{v''\leq v\\ v''s<v''}} c_{w,v}^{w,v''}\xi^{wv''}
    +\star.
  \end{align*}
  Comparing both expressions for $A_s(\zeta^w \xi^{v})$ one sees that 
  \begin{align*}
    \zeta^w \xi^{vs}&= \sum_{\substack{v''\leq vs\\ v''<v''s}}
    c_{w,v}^{w,v''s}\xi^{wv''} + \Big( \sum_{\substack{v''\leq v\\ v''s<v''}}
    c_{w,v}^{w,v''}\xi^{wv''} -\sum_{v''\leq v} c_{w,v}^{w,v''} \xi^{wv''}
    \Big)  +\star \\ 
    &=  \sum_{\substack{v''\leq vs\\ v''<v''s}} (c_{w,v}^{w,v''s}
    -c_{w,v}^{w,v''}) \xi^{wv''}+ \star. 
  \end{align*}
  Therefore,
  \[
    c_{w,vs}^{w,v''}= \begin{cases} c_{w,v}^{w,v''s}-c_{w,v}^{w,v''} &\text{if
        $v''\leq vs$ and $v''<v''s$,} \\ 0&\text{otherwise.}
    \end{cases}
  \]
  It follows from this last equality that if $c_{w,vs}^{w,v''}\ne0$, then
  $v''\leq vs$.
\end{proof}

\begin{lemma}\label{lem:be}
  Suppose $v\in W_L$. Then $\zeta^e \xi^v=\xi^v$. In other words,
  $c_{e,v}^{v}=1$ and $c_{e,v}^{v'}=0$ for $v'\in W_L$ with $v'\ne v$.
\end{lemma}

\begin{proof}
  It is convenient to use the isomorphism $\Theta$ from diagram
  \cref{ssec:alg}\cref{eq:fda} and the base change matrices in \cref{ssec:bc}. 

  Multiplication in the $S$-algebra $\bfD^*$ is pointwise (see
  \cite[\S2]{kostantkumar:equivariant}), so the identity element is
  $\sum_{z\in W} f_z$. We claim that $\sum_{z\in W} f_z= \sum_{z\in W} Y_z^*$,
  whence $\sum_{z\in W} Y_z^* Y_w^*=Y_w^*$ for $w\in W$. Assume the claim has
  been proved and recall the $S$-algebra homomorphism $i_a^*\colon \bfD^*\to
  \bfD_L^*$ in \cref{ssec:DL}\cref{eq:dl}. This mapping is injective on the
  span of $\{\, Y_v^*\mid v\in W_L\,\}$ and
  \[
    i_a^*(Z_e^* Y_v^*)= \sum_{v'\in W_L} i_a^*(Y_{v'}^*) i_a^*(Y_v^*) =
    \sum_{v'\in W_L} Y_{v',L}^*Y_{v,L}^* =Y_{v,L}^*.
  \]
  Therefore, $Z_e^* Y_v^*=Y_v^*$ for $v\in W_L$. The assertion in the lemma
  now follows because $\Theta\colon K_T(G/B)\to \bfD^*$ is a ring isomorphism.

  To prove the claim, from \cref{ssec:bc}\cref{eq:25} we have
  \[
    \delta_z= \sum_{y\in W} b_{z,y} Y_y \quad\text{and}\quad Y_y^*= \sum_{z\in
      W} b_{z,y} f_z.
  \]
  Then
  \[
    1= \delta_z\cdot 1= \sum_{y\in W} b_{z,y} (Y_y\cdot 1)=\sum_{y\in W}
    b_{z,y}, \quad\text{so}\quad \sum_{y\in W} b_{z,y}=1.
  \]
  Therefore  
  \[
    \sum_{y\in W} Y_y^*= \sum_{y\in W} \sum_{z,y\leq z\in W} b_{z,y} f_z =
    \sum_{z\in W} \big(\sum_{y,y\leq z\in W} b_{z,y} \big) f_z= \sum_{z\in W}
    f_z.
  \]
  This completes the proof.

  This claim can also be proved ``geometrically'' using the equalities
  \[
    i^*(\zeta^e)= i^*\big(\sum_{w\in W} \xi^{w} \big) = i^*([\CO_{G/B}])=
    [\CO_{P/B}]= \sum_{v'\in W_L} \xi^{v'}_L,
  \]
  where the first and last equalities follow from \cite[Lemma
  4.2]{grahamkumar:positivity}.
\end{proof}

\begin{proposition}\label{pro:indp}
  Suppose $w\in W^L$, $s=s_\alpha$ is a simple reflection such that $sw>w$ and
  $sw\in W^L$, and $v,v', v''\in W_L$. Then
  \[
    p_{swv',v}^{swv''} = \begin{cases}e^{-\alpha}
      s(p_{wv',v}^{wv''})&\text{if $v<sv$} \\
      (1-e^{-\alpha}) s(p_{wv',sv}^{wv''}) +s(p_{wv',v}^{wv''})%
                         &\text{if $sv<v$.}  \end{cases}
  \]
\end{proposition}

\begin{proof}
  This recursion formula can be proved using the coproduct defined in
  \cite[2.14]{kostantkumar:equivariant} and \cite[Proposition
  3.6]{grahamkumar:positivity}. Here we give a shorter proof using the
  isomorphism $\Theta$ from diagram \cref{ssec:alg}\cref{eq:fda} and the
  formula for structure constants in \cite[Theorem~4.1]
  {goldinzhong:structure}.

  Suppose $s_0=s$, $\alpha_0=\alpha$, $I_w=(s_1, \dots, s_p)$, and
  $I_{v''}=(s_{p+1}, \dots, s_{p+q})$. Then by
  \cite[Theorem~4.1]{goldinzhong:structure}
  \[
    p_{swv',v}^{swv''} = \sum_{\substack{E \sqsubseteq I_{swv''} \\
        \wtilde(E)=swv'}}\ \sum_{\substack{F \sqsubseteq I_{swv''} \\
        \wtilde(F)=v}} \big( B_{0}^{E,F} B_{1}^{E,F} \dotsm B_{p+q}^{E,F}
    \cdot 1 \big) b_{E,wv'} b_{F, v},
  \]
  where $b_{E, wv'}$ is the coefficient of $Y^*_{wv'}$ in $Y^*_E$, and
  similarly for $F$. It is well-known that $Y^*_J= Y^*_{\wtilde(J)}$ for any
  sequence, $J$, of simple reflections in $W$ (see \cite[Lemma 5.1]
  {goldinzhong:structure}). Thus,
  \begin{equation}
    \label{eq:p1}
    p_{swv',v}^{swv''} = \sum_{\substack{E \sqsubseteq I_{swv''} \\
        \wtilde(E)=swv'}}\ \sum_{\substack{F \sqsubseteq I_{swv''} \\
        \wtilde(F)=v}} B_{0}^{E,F} B_{1}^{E,F} \dotsm B_{p+q}^{E,F} \cdot 1 ,
  \end{equation}
  where
  \[
    B_0^{E,F}= \begin{cases}
      (1-e^{-\alpha}) \delta_{s}&\text{if $s\in E\sqcap F$,}\\
      e^{-\alpha} \delta_{s}&\text{if $s\in (E\setminus F) \sqcup
                              (F\setminus E)$,}\\
      \frac 1 {1-e^{\alpha}} - \frac {e^{-\alpha}} {1-e^{\alpha}}
      \delta_{s}&\text{if $s\notin E\sqcup F$.}
    \end{cases}
  \]
  Recall from \cref{lem:ulev}\cref{it:u3} that $I_w\sqsubseteq E$.  To
  simplify the formulas, set $\Bbar= B_{1}^{E,F} \dotsm B_{p+q}^{E,F}$, where
  the ambient sequence containing $E$ and $F$ is determined by context.
  
  Suppose first that $sv>v$. Then no reduced expression for $v$ begins with
  $s$, so for $E$ and $F$ in \cref{eq:p1}, we have $F\sqsubseteq I_{wv''}$,
  $s_0\in E \setminus F$, and $B_0^{E,F}= e^{-\alpha} \delta_{s}$. Therefore,
  \[
    B_{0}^{E,F} \Bbar \cdot 1 = \big( e^{-\alpha} \delta_{s} \Bbar \big) \cdot
    1 = e^{-\alpha} s \big( \Bbar \cdot 1\big) .
  \]
  Now
  \[
    p_{swv',v}^{swv''}  = \sum_{\substack{E \sqsubseteq I_{swv''} \\
        \wtilde(E)=swv'}}\ \sum_{\substack{F \sqsubseteq I_{wv''} \\
        \wtilde(F)=v}} e^{-\alpha} s \big( \Bbar \cdot 1\big)
    = e^{-\alpha} s \Big( \sum_{\substack{E \sqsubseteq I_{wv''} \\
        \wtilde(E)=wv'}}\ \sum_{\substack{F \sqsubseteq I_{wv''} \\
        \wtilde(F)=v}} \Bbar \cdot 1\Big) = e^{-\alpha} s (p_{wv',v}^{wv''}).
  \]
  where in the second equality we've used that the assignment $E\mapsto
  E\setminus \{s\}$ defines a bijection between $\{\, E\sqsubseteq I_{swv''}
  \mid \wtilde(E)= swv''\,\}$ and $\{\, E\sqsubseteq I_{wv''} \mid \wtilde(E)=
  wv''\,\}$.
  
  Now suppose $sv<v$. Similar reasoning shows that
  \begin{align*}
    p_{swv',v}^{swv''}%
    &=  (1-e^{-\alpha}) s\Big( \sum_{\substack{E \sqsubseteq
      I_{swv''} \\ \wtilde(E)=swv'}}\ \sum_{\substack{s_0\in F \sqsubseteq
    I_{swv''} \\ \wtilde(F)=v}}   \Bbar \cdot 1 \Big)
    + e^{-\alpha} s \Big( \sum_{\substack{E \sqsubseteq I_{swv''} \\ 
    \wtilde(E)=swv'}}\ \sum_{\substack{F \sqsubseteq I_{swv''} \setminus
    \{s_0\} \\ \wtilde(F)=v}}   \Bbar \cdot 1 \Big) \\
    &=  (1-e^{-\alpha}) s\big( p_{wv', sv}^{wv''} \big)
      + e^{-\alpha} s \big( p_{wv', v}^{wv''} \big) .
  \end{align*}
\end{proof}

\begin{corollary}\label{cor:cswv}
  With the assumptions in the proposition,
  \[
    c_{sw,v}^{sw,v''} = \begin{cases}e^{-\alpha}
      s(c_{w,v}^{w,v''})&\text{if $v<sv$,} \\
      s(c_{w,v}^{w,v''}) + (1-e^{-\alpha}) s(c_{w,sv}^{w,v''})%
                        &\text{if $sv<v$.} \end{cases}
  \]
  Consequently,
  \[
    c_{sw,v}^{sw,v}=
    \begin{cases}
      e^{-\alpha} s(c_{w,v}^{w,v})&\text{if $v<sv$} \\
      s(c_{w,v}^{w,v}) &\text{if $sv<v$.}
    \end{cases}
  \]
\end{corollary}

\begin{proof}
  The first equality follows from the proposition by summing over $v'\in
  W_L$. The second follows from the first because $c_{w,sv}^{w,v}=0$ when
  $sv<v$ by Lemma \ref{lem:coeff0}.
\end{proof}

\begin{corollary}\label{cor:lam}
  Suppose $w\in W^L$ and $v\in W_L$. Then there is an element $\lambda_{w,v}$
  in the root lattice of $\Phi$ such that $c_{w,v}^{w,v} =
  e^{-\lambda_{w,v}}$.
\end{corollary}

\begin{proof}
  The proof proceeds by induction on $\ell(w)$. By Lemma \ref{lem:be},
  $c_{e,v}^v=1$ for $v\in W_L$, and so $\lambda_{e,v}=0$. This starts the
  induction.

  Now suppose the result holds for $w$, that $sw>w$ with $sw\in W^L$, and that
  $v\in W_L$. By induction and Corollary \ref{cor:cswv},
  \begin{itemize}
  \item if $v<sv$, then $c_{sw,v}^{sw,v}= e^{-\alpha} s(c_{w,v}^{w,v}) =
    e^{-\alpha- s(\lambda_{w,v})}$, and so $\lambda_{sw,v}= \alpha+
    s(\lambda_{w,v})$, and
  \item if $sv<v$, then $c_{sw,v}^{sw,v}= s(c_{w,v}^{w,v}) = e^{-
      s(\lambda_{w,v})}$, and so $\lambda_{sw,v}= s(\lambda_{w,v})$.
  \end{itemize}
\end{proof}




\begin{thebibliography}{10}

\bibitem{bernsteingelfandgelfand:schubert}
J.~Bernstein, I.M. Gelfand, and S.I. Gelfand.
\newblock Schubert cells, and the cohomology of the spaces {$G/P$}.
\newblock {\em Uspehi Mat. Nauk}, 28(3(171)):3--26, 1973.

\bibitem{bourbaki:groupes}
N.~Bourbaki.
\newblock {\em \'{E}l\'ements de math\'ematique. {G}roupes et alg\`ebres de
  {L}ie. {C}hapitre {IV}, {V},{VI}}.
\newblock Actualit\'es Scientifiques et Industrielles, No. 1337. Hermann,
  Paris, 1968.

\bibitem{calmespetrovzainoulline:invariants}
B.~Calm\`es, V.~Petrov, and K.~Zainoulline.
\newblock Invariants, torsion indices and oriented cohomology of complete
  flags.
\newblock {\em Ann. Sci. \'{E}c. Norm. Sup\'{e}r. (4)}, 46(3):405--448 (2013),
  2013.

\bibitem{calmeszainoullinezhong:equivariant}
B.~Calm{\`e}s, K.~Zainoulline, and C.~Zhong.
\newblock Equivariant oriented cohomology of flag varieties.
\newblock {\em Doc. Math.}, (Extra vol.: Alexander S. Merkurjev's sixtieth
  birthday):113--144, 2015.

\bibitem{calmeszainoullinezhong:coproduct}
B.~Calm{\`e}s, K.~Zainoulline, and C.~Zhong.
\newblock A coproduct structure on the formal affine {D}emazure algebra.
\newblock {\em Math. Z.}, 282(3-4):1191--1218, 2016.

\bibitem{calmeszainoullinezhong:push}
B.~Calm{\`e}s, K.~Zainoulline, and C.~Zhong.
\newblock Push-pull operators on the formal affine {D}emazure algebra and its
  dual.
\newblock {\em Manuscripta Math.}, 160(1-2):9--50, 2019.

\bibitem{chrissginzburg:representation}
N.~Chriss and V.~Ginzburg.
\newblock {\em Representation theory and complex geometry}.
\newblock Birkh{\"{a}}user, Boston, 1997.

\bibitem{deodhar:characterizations}
V.V. Deodhar.
\newblock Some characterizations of {B}ruhat ordering on a {C}oxeter group and
  determination of the relative {M}{\"{o}}bius function.
\newblock {\em Invent. Math.}, 39:187--198, 1977.

\bibitem{drellichtymoczko:module}
E.~Drellich and J.~Tymoczko.
\newblock A module isomorphism between {$H^*_T(G/P)\otimes H^*_T(P/B)$} and
  {$H^*_T(G/B)$}.
\newblock {\em Comm. Algebra}, 45(1):17--28, 2017.

\bibitem{goldinzhong:structure}
Rebecca Goldin and Changlong Zhong.
\newblock Structure constants in equivariant oriented cohomology of flag
  varieties.
\newblock {\em Math. Z.}, 307(3):Paper No. 42, 27, 2024.

\bibitem{grahamkumar:positivity}
W.~Graham and S.~Kumar.
\newblock On positivity in {$T$}-equivariant {$K$}-theory of flag varieties.
\newblock {\em Int. Math. Res. Not. IMRN}, pages Art. ID rnn 093, 43, 2008.

\bibitem{hellermalagonlopez:equivariant}
J.~Heller and J.~Malag{\'o}n-L{\'o}pez.
\newblock Equivariant algebraic cobordism.
\newblock {\em J. Reine Angew. Math.}, 684:87--112, 2013.

\bibitem{HMSZ:formal}
A.~Hoffnung, J.~Malag\'{o}n-L\'{o}pez, A.~Savage, and K.~Zainoulline.
\newblock Formal {H}ecke algebras and algebraic oriented cohomology theories.
\newblock {\em Selecta Math. (N.S.)}, 20(4):1213--1245, 2014.

\bibitem{kostantkumar:nil}
B.~Kostant and S.~Kumar.
\newblock The nil {H}ecke ring and cohomology of {$G/P$} for a {K}ac-{M}oody
  group {$G$}.
\newblock {\em Adv. in Math.}, 62(3):187--237, 1986.

\bibitem{kostantkumar:equivariant}
B.~Kostant and S.~Kumar.
\newblock {$T$}-equivariant {$K$}-theory of generalized flag varieties.
\newblock {\em J. Differential Geom.}, 32(2):549--603, 1990.

\bibitem{lenartzainoullinezhong:parabolic}
C.~Lenart, K.~Zainoulline, and C.~Zhong.
\newblock Parabolic {K}azhdan-{L}usztig basis, {S}chubert classes, and
  equivariant oriented cohomology.
\newblock {\em J. Inst. Math. Jussieu}, 19(6):1889--1929, 2020.

\bibitem{levinemorel:algebraic}
M.~Levine. and F.~Morel.
\newblock {\em Algebraic cobordism}.
\newblock Springer Monographs in Mathematics. Springer, Berlin, 2007.

\end{thebibliography}
\bibliographystyle{plain}


\end{document}